\documentclass[a4paper, twoside, 11pt]{article}
\usepackage{verbatim}

\usepackage{amsmath, amssymb}
\newtheorem{theorem}{Theorem}

\newtheorem{conjecture}[theorem]{Conjecture}
\newtheorem{corollary}[theorem]{Corollary}

\newtheorem{definition}[theorem]{Definition}
\newtheorem{example}[theorem]{Example}

\newtheorem{statement}[theorem]{Statement}
\newtheorem{lemma}[theorem]{Lemma}

\newtheorem{proposition}[theorem]{Proposition}
\newtheorem{remark}[theorem]{Remark}

\newcommand{\vnotimes}{\overline{\otimes}}
\newenvironment{proof}[1][Proof]{\noindent\textbf{#1.} }{\ \rule{0.5em}{0.5em}}
\def\m{\mathcal{M}}

\def\u{\mathcal U}
\def\a{\mathcal A}

\def\C{\mathbb C}
\def\R{\mathbb R}
\def\N{\mathbb N}

\title{Multivariable Schur-Horn theorems }
\author{Pedro Massey\footnote{Partially supported by CONICET (PIP 0435/12) and FCE-UNLP (11X681)  
}  \,and Mohan Ravichandran}
\date{}

\newcommand{\Addresses}{{
  \bigskip
  \footnotesize

  Pedro Massey, \textsc{Depto Matemática, FCE-UNLP and
IAM ``Alberto Calder\'on'', CONICET, Argentina}\par\nopagebreak
  \textit{E-mail address}, Pedro Massey: \texttt{massey@mate.unlp.edu.ar}
  
  \medskip

  Mohan Ravichandran, \textsc{Department of Mathematics, Mimar Sinan Fine Arts University,
   Istanbul, Turkey}\par\nopagebreak
  \textit{E-mail address}, Mohan Ravichandran: \texttt{mohan.ravichandran@gmail.com}

}}

\usepackage{fancyhdr}
\begin{document}
\maketitle

\begin{abstract}
We prove a variety of results describing the possible diagonals of tuples of commuting hermitian operators in type $II_1$ factors. These results are generalisations of the classical Schur-Horn theorem to the infinite dimensional, multivariable setting. Our description of these possible diagonals uses a natural generalisation of the classical notion of majorization to the multivariable setting. In the special case when both the given tuple and the desired diagonal have finite joint spectrum, our results are complete. When the tuples do not have finite joint spectrum, we are able to prove strong approximate results. Unlike the single variable case, the multivariable case presents several surprises and we point out obstructions to extending our complete description in the finite spectrum case to the general case. We also discuss the problem of characterizing diagonals of commuting tuples in $\mathcal{B}(\mathcal{H})$ and give approximate characterizations in this case as well. 
\end{abstract}

\tableofcontents

\section{Introduction}

The classical Schur-Horn theorem characterizes the possible diagonals of a hermitian operator on $\mathbb{C}^d$. Indeed, fix an orthonormal basis $\{e_1,\cdots,e_d\}$ for $\mathbb{C}^d$ and let $E$ be the restriction map (acting on $d\times d$ matrices) that sends a matrix to its diagonal. Given hermitian diagonal matrices $A$ and $S$, the following are equivalent
(see \cite{BhaBoo})
\begin{enumerate}
\item There is a unitary $U$ such that $E(USU^*) = A$. 

\item The majorization relations hold : If $(\mu_1, \cdots, \mu_d)$ and $(\lambda_1,\cdots,\lambda_d)$ are the diagonal entries of $A$ and $S$ respectively (i.e. their eigenvalues), arranged in non-increasing order, then
\[\sum_{m=1}^k \mu_m \leq \sum_{m=1}^{k} \lambda_m\, , \quad   1 \leq k \leq d\,,\qquad \sum_{m=1}^d\mu_m = \sum_{m=1}^{d} \lambda_m\ .\]

 \end{enumerate}

The above line of investigation can be fruitfully extended to type II$_1$ factors; Here, the term ``diagonal of an operator''  refers to the trace preserving conditional expectation of an operator onto a masa. As pointed out by Arveson and Kadison in \cite{ArvKad},
it is natural to ask for a description of the relationship between a hermitian operator in a type II$_1$ factor and its possible diagonals. This problem admits a complete answer in terms of a concrete set of inequalities analogous to the matrix case. Given hermitian operators $A$ and $S$ in a type II$_1$ factor $\mathcal M$, majorization is defined by a system of inequalities analogous to the matrix case (see \cite{HiaiMa,Kamei}). Indeed, let $f_A$ and $f_S$ be the non-increasing rearrangements of $A$ and $S$ respectively; then, we say that $A$ is majorized by $S$, denoted $A \prec S$ if we have that
\begin{equation}\label{eq def maj in II1}
\int_0^r f_A \ dm \leq \int_0^r f_S \ dm\,, \quad 0 \leq r \leq 1,\qquad \int_0^1 f_A \ dm = \int_0^1 f_S \ dm\ .
\end{equation}
Let $\mathcal{A}$ be a masa in $\mathcal{M}$ and let $E$ denote the normal conditional expectation onto $\mathcal{A}$. 
Given a hermitian operator $S$ in $\mathcal{M}$, let $\mathcal{O}(S)$ denote the norm closure of the unitary orbit of $S$,
\[\mathcal{O}(S) := \overline{\{USU^{*} | U \in \mathcal{U}(\mathcal{M})\}}.\] 
Further, let us use the notation $S \approx T$ for hermitian operators $S$ and $T$ to denote that $S$ and $T$ are approximately unitarily equivalent, that is $T\in \mathcal O(S)$ (or equivalently $S\in \mathcal O(T)$). It is well known that $S\approx T$ iff they have the same moments or equivalently if the non-increasing rearrangements of $S$ and $T$ coincide. 

Kadison and Arveson proved in \cite{ArvKad} that if $A \in \mathcal{A}$ and $S \in \mathcal{M}$ are hermitian operators such that $E(T) = A$ for some $T\in \mathcal O(S)$ then we have that $A \prec S$ and they conjectured that the converse was true.  The Schur-Horn theorem in type II$_1$ factors \cite{SHFull} then states:
\begin{theorem}[The Schur-Horn theorem in II$_1$ factors](Ravichandran, 2012)\label{SHTHM}
Let  $\mathcal{M}$ be a type II$_1$ factor, $\mathcal{A}$ a masa in $\mathcal{M}$ and let $E$ denote the trace preserving conditional expectation onto $\mathcal A$. Given
hermitian operators $A \in \mathcal{A}$ and $S \in \mathcal{M}$ such that $A \prec S$ then, there is an element $T \in \mathcal{O}(S)$ such that 
\[E(T) = A\ .\]
Secondly, given hermitian operators $A$ and $S$ such that $A \prec S$, there is some masa $\mathcal{B}$ so that $E_{\mathcal{B}}(S) \approx A$. 
\end{theorem}

The two statements above are subtly different and neither follows from the other. The first statement is what Kadison and Arveson referred to as the Schur-Horn theorem in type II$_1$ factors. The solution in \cite{SHFull} builds upon previous work of Argerami and Massey \cite{ArMaSH,ArMaCa}, Dykema, Hadwin, Fang and Smith \cite{DFHS} and Bhat and Ravichandran \cite{BhaRav}. 

We remark that Theorem \ref{SHTHM} above implies 
the so-called carpenter theorem in type II$_1$ factors conjectured by R.V. Kadison in \cite{KPNAS1}:
\begin{theorem}[The carpenter theorem in II$_1$ factors](Ravichandran, 2012)\label{CTHM}
Let  $\mathcal{M}$ be a type II$_1$ factor, $\mathcal{A}$ a masa in $\mathcal{M}$ and let $E$ denote the trace preserving conditional expectation onto $\mathcal A$. Given a positive contraction $A \in \mathcal{A}$ there exists an orthogonal projection $P\in\mathcal M$ such that 
\[E(P) = A\ .\]
\end{theorem}

In recent work, Kennedy and Skoufranis have extended the Schur-Horn theorem above to generalise Thompson's theorem \cite{ThoSin} describing the possible singular values of a matrix with prescribed diagonal to the setting of type $II_1$ factors \cite{KenSko}. Upon the completion of this paper, we came to know that they have independently proved results in the multivariable case, analogous to ours.

Given the Schur-Horn theorem for single hermitians, a natural next problem is to investigate if there are multivariable analogues of the above results; Can the joint diagonals of tuples of commuting hermitians can be effectively described in terms of spectral relations?
There are currently no complete descriptions of this kind, even in the matrix case (see \cite{AYYHP3x3,ThompsonCount} and the reference therein); The goal of this paper is to show that this is an intriguing problem in both continuous(finite) and discrete factors.

We approach this problem by first defining a notion of multivariate majorization. In the single variable case, results of Hiai \cite{HiaiMa} show that the majorization inequalities in Eq. \eqref{eq def maj in II1} are equivalent to the existence of a doubly stochastic map $D$ (a linear, unital, positive, trace preserving map acting on $\m$) such that $D(S)=A$. With this in mind, it is natural  to define majorization for tuples as follows. 
\begin{definition}[Joint Majorization]\label{DefMaj}
Given tuples $\bold{S} = (S_1,\cdots,S_n)$ and $\bold{A} = (A_1,\cdots,A_n)$ in a type II$_1$ factor or a matrix algebra $\mathcal{M}$, say that $\bold{A}$ is majorized by $\bold{S}$ (which we will denote by $\bold{A} \prec \bold{S}$) if there is a doubly stochastic map $D$ on $\mathcal{M}$ such that $D(S_m) = A_m$ for $1 \leq m \leq n$. 
\end{definition}

In case both $\bold A$ and $\bold S$ are tuples  of commuting hermitian operators in a type II$_1$ factor, the relation $\bold A\prec \bold S$ was considered in \cite{ArMas08} (based on Hiai's work on majorization between normal operators \cite{Hiai89}), where several characterizations of this notion were shown; indeed, joint majorization between tuples of commuting hermitian operators can be characterized in terms Choquet's notion of comparison of measures and also in terms of tracial inequalities using convex functions (that we discuss in the next section). Apart from this, the fact that the set $\{\bold{A}:\bold{A} \prec \bold{S}\}$ has a pleasing topological structure (it is weak* closed as well as convex) makes this a natural notion. 

\begin{remark}[Notation] Throughout the paper we will deal with tuples of hermitians, for which we will consistently use the following notations. Suppose $\bold{S} = (S_1, \cdots, S_m)$ is a tuple of operators: the expression $\phi(\bold{S})$ where $\phi$ is a map on the algebra containing the operators will be understood to mean $\phi(\bold{S}) :=(\phi(S_1), \cdots, \phi(S_n))$. We will encounter the conditional expectation onto a masa $E$ and the trace $\tau$ used in this way on tuples repeatedly throughout this paper. Tuples of operators will always be written out in bold. We will use the notation, $\|\bold S\|:=\max\{\|S_i\|: \ 1\leq i\leq n\}$.
Additionally, we will use bold greek letters, $\pmb{\alpha}$ for instance, to refer to a tuple of scalars $\pmb{\alpha} = (\alpha_1, \cdots, \alpha_n)$. In both cases, the expression $q \bold{S}$ or $q \pmb{\alpha}$ where $q$ is a scalar (or an operator), will refer to $(q\,S_1, \cdots, q\,S_n)$ and $(q \alpha_1, \cdots, q\alpha_n)$ respectively. Finally, the expression $\bold{S} + \bold{T}$ will refer to the tuple $(S_1+T_1, \cdots, S_n+T_n)$. 
\end{remark}

\begin{example}\label{exa majorization}
Let $\mathcal A$ be a masa in the type II$_1$ factor $\mathcal M$, let $E$ be the trace preserving conditional expectation onto $\mathcal A$ and let $U\in\mathcal M$ be a unitary operator. Then, it is easy to see that $(\cdot)\mapsto E(U\,(\cdot )\, U^*)$ is a doubly stochastic map. Hence, for every
$n$-tuple $\bold{S} = (S_1,\cdots,S_n)\in\mathcal{M}^n$ 
we have that $E(U\,\bold{S}\,U^*) \prec \bold{S}$. 
\end{example}

 Let $\mathcal U(\mathcal M)$ denote the unitary group of the type II$_1$ factor $\mathcal M$ (or the algebra of $m\times m$ complex matrices). 
As in the single variable case, given $\bold{S} = (S_1,\cdots,S_n)\in\mathcal M^n$, we define 
\[\mathcal{U}(\bold{S}) := \{U\,\bold S\,U^*:\  U \in \mathcal{U}(\mathcal{M}) \}
\quad \text{ and } \quad
\mathcal{O}(\bold{S})=\overline{  \mathcal{U}(\bold{S}) }^{||}\, ,\] 
that is, the (joint) unitary orbit of $\bold{S}$ and 
the norm closed (joint) unitary orbit of $\bold{S}$, respectively.
It is easy to see that when $\bold S$ is an $n$-tuple of commuting hermitians, then
$\mathcal{O}(\bold{S})$ is not only $\|\cdot\|$-closed (by construction) but it is also closed in the SOT(strong operator topology).

\medskip

A very satisfactory multivariable Schur-Horn theorem would then be the following characterization of 
the joint diagonals of a commuting tuple of hermitians operators in a type II$_1$ factor:

\begin{statement}[Multivariable Schur-Horn]\label{conjSH1}
Let $\mathcal{M}$ be a type II$_1$ factor, $\mathcal{A}$ a masa in $\mathcal{M}$ and let $E$ denote the trace preserving conditional expectation onto $\mathcal A$. Let $\bold{A}$ and $\bold{S}$ be $n$-tuples of commuting hermitian operators with $\bold{A} \in \mathcal{A}^n$ and $\bold{S} \in \mathcal{M}^n$ such that $\bold{A} \prec \bold{S}$. Then, there is a tuple $\bold{T}$ in $\mathcal{O}(\bold{S})$ such that 
\[E(\bold{T}) = \bold{A}\ . \]
\end{statement}
Taking into account Example \ref{exa majorization}, Statement \ref{conjSH1} is equivalent to the identity
\begin{equation}\label{eq identity strong MSH}
E(\mathcal O(\bold S))=\{\bold A\in\mathcal A^n:\ \bold A\prec \bold S\}\, ,
\end{equation}
where $\bold S$ is an $n$-tuple of commuting hermitian operators in $\m$ and $E$ denotes the conditional expectation onto a masa $\mathcal A\subset\mathcal M$.
\medskip

%
Unfortunately, Statement 
\ref{conjSH1} does not hold in full generality, see section \ref{obs}. The main result in this paper is that 
Statement \ref{conjSH1}
does hold when both tuples $A$ and $S$ have finite spectrum (Theorem \ref{SHThm}). We then use this to prove the following result (see section \ref{consequences} for its proof):
\begin{theorem}[The Approximate Multivariable Schur-Horn theorem]\label{SHApp}Let $\mathcal{M}$ be a type II$_1$ factor, $\mathcal{A}$ a masa in $\mathcal{M}$ and let $E$ denote the trace preserving conditional expectation onto $\mathcal A$. 
 For any  $n$-tuple $\bold S$ of commuting hermitians in $\mathcal{M}$,
\begin{equation}\label{eqSHApp}
\overline{E(\mathcal{U}(\bold{S}))}^{||} = \{\bold{A} \in \mathcal{A}^n : \bold{A} \prec \bold{S}\}\ .
\end{equation}
\end{theorem}

\begin{remark}\label{rem comparison}
With the notations of Theorem \ref{SHApp}, notice that this result is weaker than Statement \ref{conjSH1} as a description of the set $\{\bold A\in\mathcal A^n:\ \bold A\prec \bold S\}$: to see this, compare the identities of Eqs. 
\eqref{eq identity strong MSH}
and \eqref{eqSHApp} using the fact that, by the norm continuity of $E$, we always have that 
$$E(\mathcal O (\bold S))\subset\overline{E(\mathcal{U}(\bold{S}))}^{||} \ .  $$
\end{remark}

It can be shown (see Proposition \ref{Pro}) that for every commuting tuple of positive contractions $\bold{A}$, there is a commuting tuple of projections $\bold{P}$ such that $\bold{A} \prec \bold{P}$. It is then natural to conjecture the following extension of the carpenter theorem for type II$_1$ factors (which is a particular case of Statement \ref{conjSH1} above):
\begin{conjecture}[The Multivariable carpenter theorem]\label{MVC}
Let $\mathcal{M}$ be a type II$_1$ factor, $\mathcal{A}$ a masa in $\mathcal{M}$ and let $E$ denote the trace preserving conditional expectation onto $\mathcal A$. 
If $\bold{A}$  is an $n$-tuple of positive contractions in $\mathcal{A}$ then, there is a $n$-tuple of commuting projections $\bold{P} $ such that  
\[E(\bold{P}) = \bold{A}\ .\]
\end{conjecture}

We feel that Conjecture \ref{MVC} (the multivariable carpenter theorem) does hold; Indeed, it holds when the majorized tuple has finite spectrum (Theorem \ref{MVCT}). By the previous remarks and Theorem \ref{SHApp} we get 
the following result (see section \ref{consequences} for its proof):
\begin{theorem}[Approximate multivariable carpenter theorem]\label{MVCApp} 
Let $\mathcal{M}$ be a type II$_1$ factor, $\mathcal{A}$ a masa in $\mathcal{M}$ and let $E$ denote the trace preserving conditional expectation onto $\mathcal A$. 
If $\bold{A}$  is an $n$-tuple of positive contractions in $\mathcal{A}$ and $\varepsilon >0$ then, there is a $n$-tuple of commuting projections $\bold{P} $ such that  
\[\| E(\bold{P}) -  \bold{A}\|\leq \varepsilon\   .\]
\end{theorem}

In another direction, it is well known that in $B(\mathcal H)$ 
there is a difference between approximate diagonals and true diagonals 
of selfadjoint operators. For example, the celebrated Kadison's Carpenter's theorem in $B(\mathcal H)$ in \cite{KPNAS2} shows
that there are some interesting obstructions for a sequence $(a_n)_{n\in\mathbb N}$ in $[0,1]$ such that $\sum_{n\in\mathbb N}a_n=\sum_{n\in\mathbb N}(1- a_n)=\infty$ to be the diagonal - with respect to an orthonormal basis of $\mathcal H$ - of an orthogonal projection, while the results from \cite{NeuSH} show that any such sequence in $[0,1]$ can be $\|\cdot\|_\infty$-approximated by diagonals of orthogonal projections in $B(\mathcal H)$ (i.e. any such sequence in $[0,1]$ is an approximate diagonal of orthogonal projections in $B(\mathcal H)$). Arveson extended Kadison's results in this matter (see \cite{ArDiPN}) and showed that there are obstructions (besides the expected trivial ones) for a sequence of complex numbers to be the diagonal of a normal operators with finite spectrum (see section \ref{MSAppB(H)} for details). The following result - that partially extends the results on approximate diagonals from \cite{NeuSH} - shows that for the normal operators considered in \cite{ArDiPN} there are no non-trivial obstructions for a sequence to be an approximate diagonal (see Section \ref{MSAppB(H)} for its proof).

\begin{theorem}[Approximate diagonals of some normal operators in $B(\mathcal H)$]\label{thm on appr diag of normals}
Let $\mathcal{A} \subset \mathcal{B}(\mathcal{H})$ be an atomic masa and let $\bold{S}$ be a collection of $n$ commuting hermitians with finite joint spectrum, with the points in the spectrum all having infinite multiplicity and lying on the vertices of a convex set $C_\mathcal X$ in $\mathbb{R}^n$. Then,
\[\overline{E(\mathcal U_\mathcal H(\bold S))}^{\, ||} = \{ \bold{A} \in \mathcal{A}^n_{sa}\, :  \ \sigma(\bold{A})\subset C_\mathcal X\}\, ,\]
where $\sigma(\bold{A})$ denotes the joint spectrum of the tuple of commuting hermitian operators $\bold A$. \end{theorem} 

\subsection{Obstructions to a multivariable Schur-Horn theorem}\label{obs}
We now discuss a couple of obstructions to a multivariable Schur-Horn theorem. The first, in the context of matrix algebras is due to Arveson, who observed this when analyzing the possible diagonals of normal operators in \cite{ArDiPN}. He noted that the problem of characterizing diagonals of normal operators is equivalent to characterizing the joint diagonals of a pair of commuting hermitians.

In the matrix context, given $d\times d$ normal matrices $A$ and $S$ with eigenvalues $\mu,\,\lambda\in \mathbb C^d$ respectively then $A\prec S$ (or equivalently the commuting two tuple $(\text{Re}(A),\text{Im}(A))$ is majorized by the commuting two tuple  $(\text{Re}(S),\text{Im}(S))$) if and only if there exists an $d\times d$ doubly stochastic matrix $D$ (i.e. $D$ has non-negative entries and the sum of the entries in each row and column equals 1) such that $D\lambda=\mu$ 
(see Proposition \ref{matrix} below).

 Now, let $A$ and $S$ be $3 \times 3$ normal complex matrices with eigenvalues $\mu = (\dfrac{1}{2},\dfrac{i}{2},\dfrac{1+i}{2})$ and $\lambda =(1,0,i)$. Let $D$ be the doubly stochastic matrix
\begin{equation}\label{exa ds not uni}D = \left(\begin{array}{ccc}
\frac{1}{2} & \frac{1}{2} & 0\\
0 & \frac{1}{2}& \frac{1}{2} \\
\frac{1}{2} & 0 & \frac{1}{2} \end{array} \right)\ .
\end{equation}
Then, we have that $D\lambda = \mu$. However, there is no $3\times 3$ unitary $U$ so that $E_3(USU^*) = A$, where $E_3$ is the restriction to the diagonal, as a map acting on $3\times 3$ matrices. Interestingly, by doubling the multiplicities both of the eigenvalues and the diagonal entries, we can resolve the problem. This is discussed in section \ref{secDS}.

\begin{remark}\label{rem MSH fails in matrices}
Let $\mathcal M_d$ denote the algebra of $d\times d$ complex matrices and let $\mathcal A_d\subset \mathcal M_d$ denote the 
diagonal masa in $\mathcal M_d$. Let $E_d$ be, as before, the restriction to the diagonal (i.e., the trace preserving conditional expectation onto $\mathcal A_d$).
The previous example shows that in general, for a
$n$-tuple $\bold S$ of commuting hermitian matrices in $\mathcal M_d$,
the closed sets
$$ E_d(\mathcal U(\bold S)) \quad \text{ and } \quad \{\bold A\in \mathcal A_d^n:\ \bold A\prec \bold S\}  $$
are not equal, as opposed to the one variable case ($d=1$) in which the equality of 
these sets (for every such $\bold S=S$) is equivalent to the Schur-Horn theorem for matrices. 
\end{remark}

\medskip

In light of Remark \ref{rem MSH fails in matrices} we point out that Theorem \ref{SHApp} is somewhat unexpected, as its finite dimensional (matrix) version does not hold (as opposed to Theorem \ref{SHTHM} - which is the natural extension to the type II$_1$ factor setting of the classical Schur-Horn theorem). 

\medskip

The above obstruction does not emerge in type II$_1$ factors, due to their ``diffuseness''; There is however a different kind of obstruction that arises because of this very ``diffuseness''. Let $\mathcal{R}$ be the hyperfinite type II$_1$ factor, let $\mathcal{A}$ be a masa in $\mathcal{R}$ and let $E$ denote the trace preserving conditional expectation onto $\mathcal A$. Choose a unitary $U$ that generates $\mathcal{A}$. Next, consider the masa $\mathcal{A} \vnotimes \mathcal{A}$ inside $\mathcal{R} \vnotimes \mathcal{R}$ and consider the tuple $(U \otimes I,I \otimes U)$. The map $\Delta$ from $\mathcal{R} \vnotimes \mathcal{R}$ to $\mathcal{R}$ (considered as a map from $\mathcal{R}$ to itself) given by 
\[\Delta(X \otimes Y) = \tau(Y)\, E(X)  \]
on simple tensors and extended linearly 
to all of $\mathcal{R} \vnotimes \mathcal{R}$ is a doubly stochastic map. We have that
\[\Delta(U \otimes I) = U, \qquad \Delta(I \otimes U) = \tau(U)\, I\]
that is, 
\[(U,\tau(U)\,I) \prec (U \otimes I,I \otimes U)\ .\] 
Notice that these relations correspond to a joint majorization between the $4$-tuples of commuting hermitians 
given by real and imaginary parts of the normal operators in each 2-tuple.
However, there is no element $\bold{T}=(A,B)$ in $\mathcal{O}((U \otimes I,I \otimes U))$ so that $E(\bold{T}) = (U,\tau(U)\,I)$. This is seen as follows: Let $U_n$ be a sequence of unitaries so that both $U_n (U \otimes I) U_n^*$ and $U_n (I \otimes U) U_n^*$ converge in norm to the commuting unitary operators $A$ and $B$ respectively and additionally, \[E(A) = U\, .\] We note that 
\begin{eqnarray*}
 \tau((A-U)^{*}(A-U)) &=& \tau(A\,A^{*}) - 2 \operatorname{Re}\tau(A^{*}\,U) + \tau(U\, U^{*})\\
&=&  2 - 2\operatorname{Re}\tau(E(A)^{*}\, U)=0\, .
\end{eqnarray*}
implying that $A = U$. Consequently, noting that $U$ generates $\mathcal{A}$ and that $A$ commutes with $B$, we see that $B$ also lies in $\mathcal{A}$. Hence, $E(B)=B\in \mathcal O(I\otimes U)$ cannot possibly be a scalar. This rules out Statement \ref{conjSH1}, which was the natural (straightforward) extension of the Arveson-Kadison type Schur-Horn theorem. 

\begin{remark}\label{rem diag vs approx diag}
Let $\m$ be a type II$_1$ factor, let $\a\subset \m$ be a masa and let $E$ denote the trace preserving conditional expectation onto $\a$. As a consequence of the results in \cite{SHFull}, we get that for an hermitian operator $S\in \m$ 
then $$ E(\mathcal O(S))=\overline{E(\mathcal{U}(S))}^{||}$$ i.e. diagonals and approximate diagonals of $S$ coincide.
On the other hand, the previous example shows that in the multivariable case the sets 
$$ E(\mathcal O(\bold S)) \quad \text{ and }\quad \overline{E(\mathcal{U}(\bold S))}^{||}
$$ of joint diagonals and approximate joint diagonals of $\bold S$ do not coincide (which is in accordance with the distinction between diagonals and approximate diagonals of some normal operators in $B(\mathcal H)$, see the comments before Theorem \ref{thm on appr diag of normals}). Thus, there are obstructions for an tuple $\bold A\in\a^n$ to be a joint diagonal of $\bold S$. The nature of these obstructions is not yet understood.  
\end{remark}

\subsection{Outline of the paper}
We collect a few facts about Joint majorization in section \ref{sec2}, most importantly relating it to Choquet's notion of comparison of measures. We also provide a way of concretely checking this when both tuples of operators have finite spectrum. Section \ref{sec3} contains a proof of the Schur-Horn theorem when both tuples have finite spectrum - We accomplish this by first proving the result when the majorized tuple consists of scalars; we then use the scalar case
to prove the result in case both commuting tuples have finite spectrum, using the equivalence between majorization and Choquet's notion of comparison of measures. We then collect consequences of this theorem, 
including an approximate Schur-Horn theorem for general tuples of commuting hermitians and 
an approximate multivariable carpenter theorem. In section \ref{sec4}, we adapt an idea of Dykema et al \cite{DFHS} to show that we can find diffuse abelian orthogonal subalgebras of masas (in the sense of Popa) and use this to show the partial validity of Statement \ref{conjSH1} 
in certain II$_1$ factors.
In section \ref{secDS}, we digress to 
discuss the situation in matrix algebras, obtaining some approximate representations of the action of doubly stochastic 
matrices in terms of dilations. We end this section with a simple characterization of approximate diagonals of tuples of commuting hermitians with finite spectrum (with each element in the spectrum being of infinite multiplicity).

\section{Joint majorization in type II$_1$ factors}\label{sec2}
Let $\mathcal{M}$ be a type II$_1$ factor.  Given tuples of operators $\bold{A} = (A_1,\cdots,A_n)$ and $\bold{S} = (S_1,\cdots,S_n)$ in $\mathcal{M}$, recall(Definition \ref{DefMaj}) that $\bold{A}$ is majorized by $\bold{S}$, denoted $\bold{A} \prec \bold{S}$ if there is a doubly stochastic map $D$ (unital, positive and trace preserving) from $\mathcal{M}$ to itself such that $D(\bold{S}) = \bold{A}$. 

When both $\bold A$ and $\bold S$ are tuples of commuting hermitians, the relation $\bold A\prec \bold S$ is equivalent to any of the following conditions (see \cite{ArMas08}):
\begin{enumerate}\label{EqMaj}
 \item For every convex function $f: \mathbb{R}^n \rightarrow \mathbb{R}$, we have that 
 \[\tau(f(A_1,\cdots,A_n)) \leq \tau(f(S_1,\cdots,S_n))\ .\]
 \item We have that
 \[\bold{A} \in \overline{\operatorname{conv}\{\mathcal{O}(\bold{S})\}}^{\operatorname{w}^*} = \overline{\operatorname{conv}\{U\bold{S}U^{*}:\ U\in \mathcal U(\mathcal M)\}}^{\operatorname{w}^*}\ .\]
 \item 
 Let $\mu$ and $\nu$ be the (joint) scalar spectral measures of $\bold A$ and $\bold S$ respectively; that is, for every polynomial in $n$ variables, we have that
 \[\tau(p(A_1,\cdots,A_n)) = \int_{\mathbb{R}^n} p(x_1,\cdots,x_n)\ d\mu\]
 and similarly for $\bold S$ and $\mu$. Then, for every tuple of positive Borel measures $\mu_1, \cdots, \mu_m$ such that 
 $\sum_{i=1}^{m} \mu_i = \mu$, there are positive Borel measures $\nu_1,\cdots,\nu_m$ such that $\sum_{i=1}^{m} \nu_i = \nu$ and $$\int_{\mathbb R^n} x_j \ d\nu_i = \int_{\mathbb R^n} x_j \ d\mu_i  \ , \quad 1\leq i\leq m \ , \ 1\leq j\leq n\ .$$
In this case, say that $\nu$ majorizes $\mu$ in Choquet's sense and write $\mu \prec \nu$.
\end{enumerate}

Using the description using Choquet's comparison of measures in item 3 above, we can describe joint majorization in operator algebraic language as follows
\begin{remark} Given two tuples $\bold{A}$ and $\bold{S}$ in a masa $\mathcal{A}$, we have that $\bold{A} \prec \bold{S}$ iff for every partition of the identity, $P_1 + \cdots + P_m = I$ into mutually orthogonal projections in $\mathcal{A}$, we have a partition of the identity, $Q_1 + \cdots + Q_m = I$ into mutually orthogonal projections  in $\mathcal{A}$ so that 
\[\tau(P_j) = \tau(Q_j)\quad \text{and} \quad \tau(\bold{A}P_j) = \tau(\bold{S}Q_j)\quad \text{for} \quad 1 \leq j \leq m\ .\]
\end{remark}
It is easy to see that the set $\{\bold{A} \in \mathcal{A}^n \mid \bold{A} \prec \bold S\}$ is a convex, weak* closed set.  

\begin{proposition}\label{Pro} Let $\mathcal M$ be a type II$_1$ factor. If $\bold{A}$ is a tuple of commuting positive contractions then there exists a tuple of commuting projections $\bold P$ such that $\bold A\prec \bold P$.
\end{proposition}
\begin{proof}
Take a partition $Q_1+\cdots+Q_m=I$ into mutually orthogonal projections such that for each $1\leq i\leq m$, $Q_i$ commutes with $\bold A$ and such that there exists a permutation $\sigma_i$ of $\{1,\ldots,n\}$ with $Q_i A_{\sigma_i(j)}\leq Q_i A_{\sigma_i(j+1)}$ for $1\leq j\leq n-1$. Now set $R_{i,1}=Q_i A_{\sigma_i(1)}$ and $R_{i,j}= Q_i(A_{\sigma_i(j)}-A_{\sigma_i(j-1)})$ for $2\leq j\leq n$ and $1\leq i\leq m$.
Notice that for $1\leq i\leq m$ and $1\leq j\leq n$ we have 
 $$R_{i,n+1}:= Q_i-\sum_{k=1}^nR_{i,k}=Q_i-\max\{Q_iA_1,\ldots,Q_iA_n\}\geq 0
\ \ \text{ and } \ \ Q_iA_j= \sum_{k=1} ^{\sigma_i^{-1}(j)} R_{i,k}\ .$$
Consider a partition $Q_{i,1}+\cdots+Q_{i,n+1}=Q_i$ such that $\tau(Q_{i,j})=\tau(R_{i,j})$, for $1\leq i\leq m$; hence, 
by the one variable carpenter theorem, there exist doubly stochastic maps $T_{i,j}$ acting on $\mathcal M$ such that
$T_{i,j}(Q_{i,j})=R_{i,j}$ for $1\leq i \leq m$ and $1\leq j\leq n+1$. Finally, set 
$$T(\cdot)=\sum_{i=1}^m\sum_{j=1}^{n+1} T_{i,j}(Q_{i,j} \cdot Q_{i,j}) \ \ \text{ and } \ \ P_j=\sum_{i=1}^m\sum_{k=1}^{\sigma_i^{-1}(j)} Q_{i,k} \ , \quad 1\leq j\leq n\ .$$
It is easy to see that $T$ is a doubly stochastic map acting on $\mathcal M$ and that $\bold P=(P_1,\cdots,P_n)$ is a commuting tuple of projections such that $T(\bold P)=\bold A$ so that  $\bold A\prec \bold P$.
\end{proof}

\medskip

It is now natural to conjecture the multivariable ``carpenter'' theorem, Conjecture \ref{MVC}. While we are unable to settle this, we do prove it in the case when the operators in $\bold A$ have finite joint spectrum (see Theorem \ref{MVCT}). 

\begin{remark}\label{Unit}
In analogous fashion to Proposition \ref{Pro}, any normal contraction $A$ is majorized by a unitary. If the contraction has finite spectrum, the unitary may be taken to have finite spectrum as well. 
\end{remark}

\begin{remark}\label{rem: projections are maximal} Let $\mathcal M$ be a type II$_1$ factor. We point out that the only $n$-tuples of commuting positive contractions in $\mathcal M$ that are not majorized by another $n$-tuple of  positive commuting contractions except in a trivial way (approximately unitarily equivalent tuples joint majorize each other) are those all of whose elements are projections. This is routine and we omit the proof. 
\end{remark}

 In the one variable case, majorization can be described by a pleasing set of inequalities. Given a hermitian operator $A$ with increasing rearrangement $f_A$, define the numbers 
 \[\lambda_A(t) = \sup_{\tau(P) = t} \tau(AP)=\int_0^t f_A(x)\ dx.\] Then, given two hermitian operators, $A$ and $S$, we have that $A \prec S$ iff $\lambda_A(t) \leq \lambda_S(T)$ for $t \in [0,1]$ and $\tau(A) = \tau(S)$. Note that for any $t \in [0,1]$, the functions $A \rightarrow \lambda_A(t)$ are operator convex functions.

We have been unable to find such a pleasant characterization in the multivariable case. However, when both tuples $\bold{A}$ and $\bold{S}$ have finite spectrum, it is possible to give a concrete easily checkable description of joint majorization.

%
%
%
%
%
%
%

\begin{proposition}\label{matrix}Let $\mathcal M$ be a type II$_1$ factor.
Let $\bold P=(P_1, \cdots, P_k)$  $\bold Q=(Q_1,\cdots, Q_m)$ be two tuples of mutually orthogonal projections in $\mathcal M$ summing up to $I$. Let $\pmb\alpha_i,\, \pmb\beta_j\in \R^n$ for $1\leq i\leq k$, $1\leq j\leq m$ and let  
 $\bold A,\,\bold S\in \mathcal M$ be the $n$-tuples of commuting hermitian operators given by
\[S = \sum_{i=1}^k \pmb{\alpha}_i P_i\ , \quad A = \sum_{i = 1}^m \pmb{\beta}_i Q_i\ .\]
Then, $\bold{A} \prec \bold{S}$ iff there is a $m \times k$ matrix of non-negative numbers, $D = (d_{ij})_{1 \leq i \leq m,\, 1 \leq j \leq k}$ such that 
\begin{eqnarray}\label{matr}
D \, 1_k = 1_m ,\quad q^t D = p^t \quad \text{and} \quad D\, \pmb{\alpha}  = \pmb{\beta} \, .
\end{eqnarray}
Here, $p$ is the column vector $\tau(\bold P)^t$, $q$ is the column vector $\tau(\bold Q)^t$, $\pmb{\alpha}$ is the $k \times n$ matrix  $[\pmb{\alpha}_1 \mid \cdots \mid \pmb{\alpha}_k]^t$ and $\pmb{\beta}$ is the $m \times n$ matrix  $[\pmb{\beta}_1 \mid \cdots \mid \pmb{\beta}_m]^t$ .\end{proposition}
\begin{proof}
Let $\{P_1, \cdots, P_k\}$ and $\{Q_1, \cdots, Q_m\}$ be sets of mutually orthogonal projections summing up to $I$ in $\m$ 
and such that 
\[\bold{S} = \sum_{i = 1}^{k} \pmb{\alpha}_i\, P_i ,\quad \bold{A} = \sum_{i = 1}^{m} \pmb{\beta}_i \, Q_i\, ,\]
for $n$-tuples of real numbers $\pmb{\alpha}_i$ and $\pmb{\beta}_i$. 
Assume further that $\bold A\prec \bold S$ and let $\Delta$ be a doubly stochastic map acting on $\m$ such that $\Delta(\bold S)=\bold A$. Define $$ d_{ij}=\tau(Q_j)^{-1}\ \tau(\Delta(P_i)\,Q_j)\geq 0\ , \quad 1\leq i\leq m\ , \ 1\leq j\leq k\, .$$ If we set $D = (d_{ij})_{1 \leq i \leq m,\, 1 \leq j \leq k}$ then $D$ satisfies the conditions in Eq.  \eqref{matr}. Conversely, given a matrix $D$ satisying Eq. \eqref{matr} we set 
$$ \Delta(T)=\sum_{i=1}^m\,\sum_{j=1}^k d_{ij}\ \tau(P_i\,T\,P_i)\, Q_j \ , \quad T\in \m\,.$$
It is straightforward to check that $\Delta$ is a doubly stochastic map acting on $\m$, such that 
$\Delta(\bold S)=\bold A$.
\end{proof}

\medskip

Note that when all the projections have the same trace, this reduces to the statement that the matrix $D$ is doubly stochastic. The condition in Proposition \ref{matrix} above can be checked using linear programming. This result has a simple and pleasing consequence in the very special case when the majorizing $n$-tuple has finite joint spectrum which lies on the vertices of a simplex in $\mathbb{R}^n$. 
\begin{proposition}\label{prop joint simplex}
Let $\bold{S}$ be a $n$-tuple of commuting hermitians in a type II$_1$ factor such that the joint spectrum consists of $n+1$ points that form a simplex in $\mathbb{R}^n$. Then 
\begin{equation}\label{eq two sets}
\{\bold{A} \in \mathcal{M}_{sa}^n :\  \bold{A} \prec \bold{S}\} = \{\bold{A} \in \mathcal{M}_{sa}^n : \ \sigma(\bold{A}) \subset \operatorname{conv}(\sigma(\bold{S})), \, \tau(\bold{A}) = \tau(\bold{S})\}\ .
\end{equation}
\end{proposition}
\begin{proof}
It follows from item 3 of the list of equivalent conditions to joint majorization (at the beginning of this section)
that the set to the left in Eq. \eqref{eq two sets} is included in the set to the right. This holds in general, without conditions on the spectrum. As to the converse, first assume that the $n$-tuple of hermitians $\bold{A}$ has finite spectrum along with $\sigma(\bold{A}) \subset  \operatorname{conv}(\sigma(\bold{S}))$ and $\tau(\bold{A}) = \tau(\bold{S})$. Let us write

\[\bold{S} = \sum_{i = 1}^{n+1} \pmb{\alpha}_i P_i ,\quad \bold{A} = \sum_{i = 1}^{m} \pmb{\beta}_i Q_i\ .\]
The $\pmb{\alpha}_i$ form the vertices of a simplex $\mathcal{X}$ and by the condition on the spectrum, the $\pmb{\beta}_i$ belong to $\mathcal{X}$. 
The condition on the spectra implies that each $\pmb{\beta}_i$ can be uniquely written as the convex combination of the $\pmb{\alpha}_i$, 
\begin{eqnarray}\label{cond2}
\pmb{\beta}_i = \sum_{j=1}^{n+1} d_{ij}\,\pmb{\alpha}_j, \quad \sum_{j=1}^{n+1} d_{ij} = 1, \quad 1 \leq i \leq m.
\end{eqnarray}
 Let $D$ be the matrix $D = (d_{ij})_{1 \leq i \leq m, 1 \leq j \leq n+1}$. We see that
 \begin{eqnarray}\label{cond3}
 \sum_{i = 1}^{m} \pmb{\beta}_i \, \tau(Q_i) &=&  \sum_{i = 1}^{m} \sum_{j=1}^{n+1} d_{ij} \tau(Q_i)\,\pmb{\alpha}_j,
 =   \sum_{j=1}^{n+1}  \pmb{\alpha}_j
 \sum_{i = 1}^{m}  d_{ij} \, \tau(Q_i)  \ .\end{eqnarray}
 The trace condition implies that 
\begin{eqnarray}\label{cond1}
\sum_{j= 1}^{n+1} \pmb{\alpha}_j\, \tau(P_j) = \sum_{j = 1}^{m} \pmb{\beta}_j \, \tau(Q_j)\ .
\end{eqnarray}
  Combining Eqs. (\ref{cond2}) and (\ref{cond1}), we see that
 \begin{eqnarray}\label{cond4}
 \sum_{j= 1}^{n+1} \pmb{\alpha}_j \tau(P_j) = \sum_{j=1}^{n+1}  \pmb{\alpha}_j 
 \sum_{i = 1}^{m}  d_{ij} \, \tau(Q_i)\, .  \end{eqnarray}
  Since the convex combinations that realise the $\pmb{\beta}_i$ using the $\pmb{\alpha}_i$ are unique, we see that $D\, q = p$ where $p$ is the column vector $[\tau(P_1),\cdots,\tau(P_{n+1})]^t$, $q$ is the column vector $[\tau(Q_1),\cdots,\tau(Q_m)]^t$. Together with 
 Eq. (\ref{cond2}), we see that condition Eq. (\ref{matr}) in Proposition \ref{matrix} is satisfied and we conclude that $\bold{A} \prec \bold{S}$. 
 
 The general case (when $\bold{A}$ does not have finite joint spectrum) follows by a routine approximation argument which can be found in the proof of Theorem \ref{SHApp}, for instance (see Section \ref{consequences}).
\end{proof}

\medskip

Proposition \ref{prop joint simplex} applies to normal operators $N\in\m$ with spectrum consisting of three non-collinear points, by considering the associated 2 tuple $(\text{Re}(N),\,\text{Im}(N))$. The above result fails when the joint spectrum does not lie on the vertexes of a simplex. We use the following example taken from \cite[Example 2]{WMM} which in turn was inspired by an example of Alfred Horn. Let $\{P_i: 1 \leq i \leq 4\}$ be projections, all of trace $\frac{1}{4}$ in $\mathcal{M}$ and let $N$ and $A$ be the normal operators with spectral projections $P_i$ and spectra,
\[\sigma(N) = \{0,4i,3-2i,-3-2i\}, \quad \sigma(A) = \{2, -2, 2i, -2i\}\,.\] 
Clearly, the operators have the same trace and that $\sigma(A) \subset \operatorname{conv}(\sigma(N))$. Nevertheless, it is shown loc.cit. that $A \not\prec N$. While the points on the spectrum of $N$ above do not have the property that every triple is non-collinear, one can perturb them slightly to get a counterexample with that additional property as well.

\section[Multivariable Schur-Horn theorems ]{Multivariable Schur-Horn theorems in type II$_1$ factors}\label{sec3}

In this section we obtain an exact multivariable a characterization of joint majorization between 
tuples of commuting hermitian operators with finite spectrum (Theorem \ref{SHThm}), which is an extension of the Schur-Horn theorem in type II$_1$ factors. We then use this result to obtain a proof of the approximate multivariable Schur-Horn theorem (Theorem \ref{SHApp}). We also consider additional hypothesis on the inclusion $\a\subset \m$, where $\a$ is a masa in the type II$_1$ factor $\m$, in order to obtain partial extensions of Theorem \ref{SHThm}.

\subsection{Finite spectra: Scalar diagonals}

We  consider first the simplest case :  $\bold{S}$ is a commuting tuple of hermitians with finite joint spectrum and $\bold{A} = \tau(\bold{S})\, I$. Write $S = \sum_{i=1}^{k} \pmb{\alpha}_i P_i$. In this case, the fact that $\tau(\bold{S}) \, I \prec \bold{S} $ can by seen by using Proposition \ref{matrix} and the matrix $D = [\tau(P_1),\cdots,\tau(P_k)]$. Or, for that matter, by the fact that the map $\cdot \mapsto \tau(\cdot)\,I$ is doubly stochastic. We will prove:

\begin{proposition}\label{Scalar}
Let $\mathcal{A}$ be a masa in a type II$_1$ factor $\mathcal{M}$ and let $E$ denote the trace preserving conditional expectation onto $\mathcal A$. Let $k$ be a natural number and let $\bold{S}$ be a tuple of commuting hermitians with joint spectrum consisting of at most $k$ points. Then, there is a unitary so that $E(U\bold{S}U^*) = \tau(\bold{S})\,I$. 
\end{proposition}

We will prove this proposition using a double induction argument. We need the following definition:

\begin{definition}
Let $\bold{S}$ and $\bold{T}$ be two tuples of commuting hermitian operators with finite joint spectrum in type II$_1$ factors $\mathcal{M}$ and $\mathcal{N}$ respectively. Given a natural number $k$ we say that $\bold{S}$  $k-$\textbf{resembles} $\bold{T}$ if we may write 
\[\bold{S} = \sum_{i=1}^{k}\pmb{\alpha}_i\, Q_i, \quad \bold{T} = \sum_{i=1}^{k} \pmb{\beta}_i \,R_i, \quad \tau_{\mathcal M}(Q_i) = \tau_{\mathcal N}(R_i), \quad 1 \leq i \leq k\ ,\] where $\{Q_1,\cdots,Q_k\}$ (respectively $\{R_1,\cdots,R_k\}$) are mutually orthogonal projections that form a partition of $I_\mathcal M$ (respectively for $I_\mathcal N$).
\end{definition}


The proof of Proposition \ref{Scalar} will require the following lemma. 

 \begin{lemma}\label{Ind}
 Let $\mathcal{A}$ be a masa in a type II$_1$ factor $\mathcal{M}$ and let $E$ denote the trace preserving conditional expectation onto $\mathcal A$.  Let $P$ be a projection in $\mathcal{A}$ and let $\bold{S}$ and $\bold{T}$ be two tuples of commuting  hermitians in $P\mathcal{M}P$ and $(I-P)\mathcal{M}(I-P)$ respectively, so that $\bold{S}$  $(k-1)-$\textbf{resembles} $\bold{T}$. Then, there is a unitary $U$, a projection $Q$ in $\mathcal{A}$ of trace at least $\frac{1}{3}$ and a projection $R$ orthogonal to $Q$ so that letting $\bold{A} = \bold{S} \oplus \bold{T}$ and $\bold{B} = U\,(\bold{S}\oplus \bold{T})\,U^*$, we have,
\begin{enumerate}
\item $E(Q\,\bold{B}\,Q) = \tau(\bold{A})\, Q$.
\item $R$ commutes with $(I-Q)\,\bold{B}\,(I-Q)$ inside $(I-Q)\mathcal{M}(I-Q)$. 
\item  
$R\,\bold{B}\,R$ inside $R\mathcal{M}R$ $(k-1)-$resembles $(I-Q-R)\,\bold{B}\,(I-Q-R)$ inside $(I-Q-R)\mathcal{M}(I-Q-R)$.  
\item $\tau(\bold A)=\tau_{(I-Q)\mathcal{M}(I-Q)}((I-Q)\,\bold B\,(I-Q))$.
\end{enumerate}
\end{lemma}

A pictorial description of this result might be useful to the reader. Lemma \ref{Ind} above states that there is a unitary $U$ so that 
\[U \left(\begin{array}{cc}
\bold{S} & 0\\
0 & \bold{T}
  \end{array} \right)U^{*} = \left(\begin{array}{ccc}
\bold{X} & \ast & 0\\
\ast & \bold{Y} & 0\\
0 & 0 & \bold{Z}
  \end{array} \right)\]
where $\bold{Y}$ resembles $\bold{Z}$ and $\bold{X}$ has conditional expectation equal to the trace of $\bold{A}$. 

\medskip

Concerning Proposition \ref{Scalar} there is nothing to prove when $k = 1$; Similarly, Lemma \ref{Ind} trivially holds
in case $k-1=1$. We will prove Lemma \ref{Ind} assuming the truth of Proposition \ref{Scalar} for $1, \cdots, k-1$. Further, we will prove Proposition \ref{Scalar} for a fixed $k$ assuming the truth of Lemma \ref{Ind} for $1, \cdots, k-1$. 
(this double induction argument allows to conclude the truth of both Proposition \ref{Scalar}
and Lemma \ref{Ind}); Briefly, our argument is as follows: consider the notations of Proposition \ref{Scalar}. We may, after conjugating by a unitary, assume that $\bold{S}$ lies in the masa $\mathcal{A}$.  Suppose $\bold{S}$ has joint spectrum consisting of $k$ points. Write $\bold{S} = \sum_{i=1}^{k} \pmb{\alpha}_iP_i$.  
Assuming Lemma \ref{Ind} holds for tuples of commuting hermitians with joint spectrum of at most $k-1$ points, we will construct sequences of unitary operators $U_n$ and projections $Q_n$ in $\mathcal M$ that implement partial solutions, that is, we will have that \[E(U_n \, \bold{S}\, U_n^*) = \tau(\bold{S}) \, Q_n\, .\]
Coupled with additional facts about $Q_n$, we will show that $U_n \bold{S} U_n^*$ converge in the SOT to a tuple of commuting hermitians $\bold{W}$, say, such that 
\[E(\bold{W}) = \tau(\bold{S})\,  I \ \ \text{ and } \ \ \bold S\approx \bold W\, .\]
Therefore $\bold W$ has joint spectrum consisting of $k$ points, so there exists a unitary operator $U\in \mathcal M$ such that $\bold W=U\,\bold S\,U^*$ and then Proposition \ref{Scalar} holds for $k$ (for details see below).

Now, for the construction of the $U_n$ and $Q_n$ we need:

\medskip

\begin{proof}[Proof of Lemma \ref{Ind}](Assuming Proposition \ref{Scalar} for $k-1\geq 2$) 
After conjugating by a unitary, we may assume that $\bold{S}$ and $\bold{T}$ belong to the masas $\mathcal{A}P$ and $\mathcal{A}(I-P)$ of the type II$_1$ factors $P\m P$ and $(I-P)\m(I-P)$ respectively. We may assume that $\tau(P) \leq \frac{1}{2}$; else, we just switch the roles of $\bold{S}$ and $\bold{T}$. Let $m\geq 1$ be the integer such that 
\begin{equation}\label{choice}
(m+1)\, \tau(P) \leq 1 < (m+2)\, \tau(P)\, .
\end{equation}
Given the hypotheses, we may write 
\[\bold{S} = \sum_{i=1}^{k-1} \pmb{\alpha}_i \, E_i\ ,\quad \bold{T} = \sum_{i=1}^{k-1} \pmb{\beta}_i \, F_i\, ,\]
where $\{E_1,\cdots,E_{k-1}\}$ and $\{F_1,\cdots,F_{k-1}\}$ are mutually orthogonal projections
that form partitions of $P$ and $I-P$ respectively and such that 
\begin{equation}\label{rel traces}
\dfrac{\tau(E_i)}{\tau(P)}=\tau_{P\mathcal{M}P}(E_i) = \tau_{(I-P)\mathcal{M}(I-P)}(F_i)=\dfrac{\tau(F_i)}{\tau(I-P)}\, ,\quad 1 \leq i \leq k-1\, .
\end{equation}
For each $1 \leq i \leq k-1$, pick a family $\{F_i^1,\cdots,F_i^{m+1}\}$ of mutually orthogonal sub projections of $F_i$ such that, for $1 \leq i \leq k-1$:
\[\tau(F_i^j) = \dfrac{\tau(P)}{\tau(I-P)} \, \tau(F_i)\, ,\ \,  1 \leq j \leq m,\quad  \tau(F_i^{m+1}) = \dfrac{1-(m+1)\,\tau(P)}{\tau(I-P)}\,\tau(F_i)\ .\] Thus, $\{F_i^1,\cdots,F_i^{m+1}\}$ forms a partition of $F_i$ for each $1 \leq i \leq k-1$. Now define
\[Q_j = F_1^j + \cdots + F_{k-1}^{j},\quad 1 \leq j \leq m+1\,.\]
Note that for $1 \leq j \leq m$,
\[\tau(Q_j) = \sum_{i=1}^{k-1} \tau(F_i^j) = \dfrac{\tau(P)}{\tau(I-P)} \sum_{i=1}^{k-1} \tau(F_i) = \tau(P)\,,\]
and similarly,
\[\tau(Q_{m+1}) = 1 - (m+1)\tau(P)\, .\]
Now, for $1\leq j\leq m$ look at the operator tuple $$\bold{T}Q_j=\sum_{i=1}^{k-1} \pmb{
\beta}_i F_i^j \quad \text{where}\quad \tau(F_i^j)\stackrel{\eqref{rel traces}}{=} \tau(E_i) \ , \quad 1\leq i\leq k-1\ .$$
Pick partial isometries $V_1, \cdots, V_m$ in $\mathcal M$ so that $V_i E_j V_i^* = F_j^i$ for $1\leq i\leq m$ and $1\leq j\leq k-1$; in particular, we get that $V_i P V_i^* = Q_i$. With this in hand, we may write $\bold{S} \oplus \bold{T}$ as
\[\bold{A} = \left(\begin{array}{ccccc}
\sum_{i=1}^{k-1} \pmb{\alpha}_i \,E_i & 0 & 0 & 0 & 0\\
0 & \sum_{i=1}^{k-1} \pmb{\beta}_i\, E_i & 0 & 0 & 0\\
0 & 0 & \ast & 0 & 0\\
0 & 0 & 0 & \sum_{i=1}^{k-1} \pmb{\beta}_i \,E_i & 0\\
0 & 0 & 0 & 0 & \bold{T}_1  \end{array} \right) \]
where $\bold{T}_1$ is the compression of $\bold{T}$ to $Q_{m+1}\mathcal{M}Q_{m+1}$ i.e.
$$\bold{T}_1=\sum_{i=1}^{k-1} \pmb{\beta}_i\, F_i^{m+1}  \quad \text{and} \quad \tau(F_i^{m+1})=\frac{1-(m+1)\,\tau(P)}{\tau(P)}\,\tau(E_i) \ , \quad 1\leq i\leq k-1\ .$$ Hence, $\bold{T}_1$ in $Q_{m+1}\mathcal{M}Q_{m+1}$ $(k-1)$-resembles $\sum_{i=1}^{k-1} \pmb{\beta}_i\, E_i$ in $P \mathcal{M}P$.
Also notice that 
\begin{equation}\label{eq rel traces 2}
\tau(T)=\frac{\tau(I-P)}{\tau(P)}\,\tau(\sum_{i=1}^{k-1}\pmb \beta _i \, E_i)\implies
\tau(\bold A)=\tau(\sum_{i=1}^{k-1} \pmb{\alpha}_i\, E_i )+ \frac{\tau(I-P)}{\tau(P)}\,\tau(\sum_{i=1}^{k-1}\pmb \beta _i \, E_i)\, .
\end{equation}
Choose a $\theta$ so that $\operatorname{cos}^2(\theta) = \tau(P)$. After conjugating by the unitary 
\[W_1 = \left(\begin{array}{ccccc}
\operatorname{cos}(\theta) & \operatorname{sin}(\theta) & 0 & 0 & 0\\
-\operatorname{sin}(\theta) & \operatorname{cos}(\theta) & 0 & 0 & 0\\
0 & 0 & I & 0 & 0\\
0 & 0 & 0 & I & 0\\
0 & 0 & 0 & 0 & I  \end{array} \right)\]
we get that
\[W_1\bold{A}W_1^* = \left(\begin{array}{cccccc}
\bold{C} & \ast & 0 & 0 & 0 & 0\\
\ast & \bold{A}_2 & 0 & 0 & 0 & 0\\
0 & 0 & \sum_{i=1}^{k-1} \pmb{\beta}_i\, E_i & 0 & 0 & 0\\
0 & 0 & 0 & \ast & 0 & 0\\
0 & 0 & 0 & 0 & \sum_{i=1}^{k-1} \pmb{\beta}_i\, E_i  & 0\\
0 & 0 & 0 & 0 & 0 & \bold{T}_1  \end{array} \right) \]
where 
$$\bold C=\sum_{i=1}^{k-1}(\tau(P)\,\pmb\alpha_i + \tau(I-P)\,\pmb\beta_i)\, E_i
\quad \text{and} \quad
\bold A_2= \sum_{i=1}^{k-1}(\tau(I-P)\,\pmb\alpha_i + \tau(P)\,\pmb\beta_i)\, E_i\ . $$
Hence, 
$\tau(\bold C)\stackrel{\eqref{eq rel traces 2}}{=}\tau(P)\, \tau(\bold A)$ 
 and $\bold{A}_2$ $(k-1)-$resembles $\sum_{i=1}^{k-1} \pmb{\beta}_i\, E_i$ as well as $\bold{T}_1$. Similarly, 
\[\bold{C} = \frac{\tau(P)}{\tau(I-P)}\, \bold{A}_2+ \frac{1 - 2\tau(P)}{\tau(I-P)}\ \sum_{i=1}^{k-1} \pmb{\beta}_i\, E_i 
\]
and thus 
continuing as above, $m-1$ more times, we get a unitary operator $W=W_m\,W_{m-1}\cdots W_1\in\mathcal M$ so that 
\begin{equation}\label{eq matrix form}W\bold{A}W^* = \left(\begin{array}{cccccc}
\bold{C} & \ast & \ast & \ast & \ast & 0\\
\ast & \bold{C} & \ast & \ast & \ast & 0\\
\ast & \ast & \ast & \ast & \ast & \ast\\
\ast & \ast & \ast & \bold{C} & \ast & 0\\
\ast & \ast & \ast & \ast & \bold{A}_{m+1} & 0\\
0 & 0 & 0 & 0 & 0 & \bold{T}_1  \end{array} \right) 
\end{equation}
where $\tau(\bold{C})=\tau(P)\,\tau(\bold A)$ and $\bold{A}_{m+1}$ $(k-1)-$resembles $\bold{T}_1$. 
Notice that Eq. \eqref{eq matrix form} implies that 
$$ P(W\bold A W^*)P=\bold C, \qquad 
Q_i(W\bold A W^*)Q_i= \sum_{i=1}^{k-1}(\tau(P)\,\pmb\alpha_i + \tau(I-P)\,\pmb\beta_i)\, F_i^j \ , \ \ 1\leq i\leq m-1\ .
$$ 
In particular, each of these (compressed) tuples are hermitian commuting operators and such that
 \[\tau(P(W\bold A W^*)P)=\tau(Q_i(W\bold A W^*)Q_i)=\tau(P)\,\tau(\bold A), \quad 1\leq i\leq m-1.\]
Now, we can apply Proposition \ref{Scalar} (for $k-1$) to $P(W\bold A W^*)P$ 
in the type II$_1$ factor $P\mathcal MP$ (notice that its joint spectrum has $k-1$ points) 
and conclude that there exists a unitary operator $U_0$ in $P\mathcal MP$ such that 
\begin{equation}\label{eq spreading1} E_{\mathcal AP}(U_0\, P(W \, \bold A \,W^*)P\,U_0^*)=\tau_{P\mathcal MP}(P(W \, \bold A \,W^*)P)\, P= \tau(\bold A)\,P\, .
\end{equation}
Similarly, by applying Proposition \ref{Scalar} (for $k-1$), we conclude that for $1\leq i\leq m-1$ there exists a unitary operator $U_i\in Q_i\mathcal M Q_i$ such that  
\begin{equation}\label{eq spreadingi} E_{\mathcal AQ_i}(U_i Q_i(W \, \bold A \,W^*)Q_i\,U_i^*)=\tau_{Q_i\mathcal MQ_i}(Q_i(W \, \bold A \,W^*)Q_i)\, P= \tau(\bold A)\,Q_i\, .
 \end{equation} 
Let $Q = P + Q_1 + \cdots + Q_{m-1}$ and notice that $\tau(Q)=m\tau(P)\geq \frac{1}{3}$ because of Eq. \eqref{choice}.
Consider the (block diagonal) unitary operator $V=U_0\oplus  \cdots \oplus U_{m-1}\oplus (I-Q)\in \mathcal M$ and the unitary operator $U=VW\in\mathcal M$. Since $V$ is block diagonal then 
\[ P(U\bold A U^*)P=U_0\,P(W\bold AW^*)P\,U_0^*, \qquad  Q_i(U\bold A U^*)Q_i=U_i\,Q_i(W\bold AW^*)Q_i\,U_i^*,\ 1\leq i\leq m.\]
Item 1 of the statement now follows from
Eqs. \eqref{eq spreading1} and \eqref{eq spreadingi} and the previous identities. 

Take $R=Q_m$ and notice that item 2 of the statement follows from 
Eq. \eqref{eq matrix form}.
Item 3 follows by construction of $A_{m+1}$ and $T_1$.
Finally, since $\tau(\bold C)=\tau(P)\,\tau(\bold A)$ then
\[ \tau_{(I-Q)\mathcal{M}(I-Q)}((I-Q)\,\bold B\, (I-Q)) =\dfrac{\tau(\bold A)-m\cdot \tau(\bold C)}{1-m \,\tau(P)}=\tau(\bold A)\, . \] 
\end{proof}

\begin{proof}[Proof Of Proposition \ref{Scalar}](Assuming Lemma \ref{Ind} for $k-1\geq 1$).
After conjugating by a unitary, we may assume that $\bold{S}$ lies in $\mathcal{A}$ and we may write $\bold{S} = \sum_{i=1}^{k} \pmb{\alpha}_i \, P_i$ for mutually orthogonal projections that form a partition of the identity $\{P_1,\cdots,P_k\}$ in $\mathcal{A}$. 
Set $P=P_1+\cdots+P_{k-1}$ so that $I-P=P_k$.
Since $$ \sum_{i=1}^{k-1} \frac{\tau(P_i)\ \tau(I-P)}{\tau(P)}=\tau(P_k)$$
we can pick mutually orthogonal projections $\{F_1,\cdots,F_{k-1}\}$ that form a partition of $P_k$ and such that 
\begin{equation}
\tau(F_i)=\frac{\tau(P_i)\ \tau(I-P) }{\tau(P)} \quad \text{ for } 1\leq i\leq k-1\ .
\end{equation}
Hence, $(\sum_{i=1}^{k-1}\pmb\alpha_i\,P_i)$ and $(\sum_{i=1}^{k-1}\pmb\alpha_k\,F_i)$ are two tuples of commuting hermitians in $P\mathcal M P$ and $(I-P)\mathcal M (I-P)$ respectively, so that they $(k-1)$ resemble each other.
Apply Lemma \ref{Ind} (for $k-1$) to get a unitary $U_1$ and projections $Q_1$ and $R_1$ in $\mathcal M$, so that letting $\bold{S}_1 = U_1 \,\bold{S}\, U_1^*$
\[E(Q_1 \bold{S}_1 Q_1) = \tau(\bold{S}) \, Q_1\ , \quad \tau(Q_1) \geq \frac{1}{3}\]
and also, $\bold{T} = (I-Q_1) \,\bold{S}_1\, (I-Q_1) $ commutes with $R_1$ inside $(I-Q_1) \mathcal{M}(I- Q_1)$; therefore, we can apply  Lemma \ref{Ind} again to $(\bold{T}\,R_1, \bold{T}\,(I-Q_1-R_1))$ (for $(k-1)$) inside $(I-Q_1)\mathcal{M}(I-Q_1)$. In this way, we get sequences of unitary operators $U_n$ and projections $Q_n$ in $\mathcal M$,  so that letting $\bold{S}_n = U_n \, \bold{S}\, U_n^*$, we have $Q_n \, \bold{S}_n \, Q_n = Q_n \, \bold{S}_m \, Q_n$ for $m > n$ and (by item 4 of Lemma \ref{Ind})
\[E(Q_n \bold{S} Q_n) = \tau(\bold{S})\, Q_n\ , \quad \tau(Q_n) \geq 1 -  \left(\frac{2}{3}\right)^n\ .\]
These facts show that the sequence of $\bold{S}_n$'s converge in SOT to a tuple $\bold W$ of commuting hermitian operators in $\mathcal M$. Since $\mathcal O(\bold S)$ is SOT closed we see that $\bold W\in\mathcal O(\bold S)$; moreover, since 
$\bold S$ has finite joint spectrum we get that $\mathcal O(\bold S)=\{V\,\bold S\, V^*: V\in\mathcal U(\mathcal M) \}$ and there exists a unitary $U\in\mathcal M$ such that $\bold W=U\,\bold S\,U^*$. Finally, by continuity of $E$ we see that $U$ has the desired properties.
\end{proof}

\subsection{Finite spectra: exact multivariable Schur-Horn theorem}

Let $\mathcal{M}$ be a type II$_1$ factor, let $\mathcal{A}$ be a masa in $\mathcal{M}$ and let $E$ denote the trace preserving conditional expectation onto $\mathcal A$. 
Suppose $\bold{S}$ is a commuting tuple, $U\in\mathcal M$ is a unitary operator and $E(U\,\bold{S}\,U^*) = \bold{A}$. Then, since the map $(\cdot)\mapsto E(U\,(\cdot )\,U^*)$ is doubly stochastic, we have that $\bold{A} \prec \bold{S}$.
 Our goal in this subsection is to prove the converse of this fact when $\bold A$ and $\bold S$ have finite spectrum.
\begin{theorem}\label{SHThm}
Let $\mathcal{M}$ be a type II$_1$ factor, let $\mathcal{A}$ be a masa in $\mathcal{M}$ and let $E$ denote the trace preserving conditional expectation onto $\mathcal A$. Let $\bold{A}$ and $\bold{S}$ be $n$-tuples of commuting hermitian operators with finite spectrum with $\bold{A} \in \mathcal{A}^n$ and $\bold{S} \in \mathcal{M}^n$ such that $\bold{A} \prec \bold{S}$. Then, there is a unitary $U$ in $\mathcal{M}$ so that 
\[E(U\bold{S}U^{*}) = \bold{A}\,.\]
\end{theorem}
\begin{proof}
We may assume, after conjugating by a unitary if needed, that $\bold{S}\in \mathcal{A}^n$. Let $\bold{A} = \sum_{i = 1}^m \pmb{\beta}_i Q_i$, where $\{Q_1,\cdots,Q_m\}$ are mutually orthogonal projections in $\mathcal{A}$ that form a partition of the identity. Using the equivalence between joint majorization and Choquet's notion of comparison of measures we see that there are orthogonal projections $R_1,\cdots, R_m$ in $\mathcal{A}$ so that 
\[\tau(R_i) = \tau(Q_i)\ , \quad  \tau(\bold{S}R_i) = \pmb{\beta}_i\,\tau(R_i)\ , \quad 1 \leq i \leq m\,.\]By Proposition \ref{Scalar} we may find a unitary $U_i$ in each of the type II$_1$ factors $R_i \mathcal{M} R_i$
 so that 
\[E_{\mathcal{A}R_i}(U_i \bold{S}R_i U_i^{*}) = \pmb{\beta}_i I_{R_i \mathcal{M} R_i} \ , \quad 1 \leq i \leq m\, .\]
Letting $U$ be the unitary $U_1 \oplus \cdots \oplus U_m\in \mathcal M$, we see that 
\[E_{\mathcal{A}}(U\bold{S}U^{*}) = \sum_{i=1}^m \pmb{\beta}_i Q_i = \bold{A}\, .\]
 \end{proof}

A simple consequence of the above theorem is the following result, which deserves to be called the multivariable Carpenter theorem (for commuting tuples of positive contractions with finite spectrum). As pointed out in the introduction (see Remark \ref{rem: projections are maximal}), the $n$-tuples of commuting projections are $\prec$-maximal elements, modulo approximate unitary equivalence, within the set of $n$-tuples of commuting positive contractions in a type II$_1$ factor. 

\begin{theorem}[The multivariable carpenter theorem: finite spectrum case]\label{MVCT}
Let $\mathcal{M}$ be a type II$_1$ factor, let $\mathcal{A}$ be a masa in $\mathcal{M}$ and let $E$ denote the trace preserving conditional expectation onto $\mathcal A$. For every $n$-tuple 
$\bold{A}\in \mathcal{A}^n$ of positive contraction with finite spectrum there is an $n$-tuple of commuting projections $\bold{P}$ such that 
\[E(\bold{P}) = \bold{A}\, .\]
\end{theorem}
\begin{proof} By Proposition \ref{Pro} there exists an $n$-tuple of commuting projections $\bold{Q}$ such that $\bold A\prec \bold Q$. Since $\bold Q$ also has a finite spectrum we can apply Theorem \ref{SHThm} and get a unitary operator $U\in\mathcal M$ such that $E(U\,\bold Q\,U^*)=\bold A$. Then $\bold P=U\,\bold Q\,U^*$ has the desired properties.
\end{proof}

\medskip

Next we describe the set of all possible diagonals of $n$-tuples of mutually orthogonal projections.

\begin{proposition}
Let $\mathcal{M}$ be a type II$_1$ factor, let $\mathcal{A}$ be a masa in $\mathcal{M}$ and let $E$ denote the trace preserving conditional expectation onto $\mathcal A$. Suppose $\bold{A}=(A_1,\cdots,A_n)\in\mathcal A^n$ is a $n$-tuple of positive operators with finite spectrum. Then there is a tuple of mutually orthogonal projections  $\bold{P}$ such that $E(\bold{P}) = \bold{A}$ iff $\sum_{i=1}^{n} A_i \leq I$. 
\end{proposition}
\begin{proof}
If we assume that $\bold{P}=(P_1,\cdots, P_n)$ is an $n$-tuple of mutually orthogonal projections such that $E(\bold{P}) = \bold{A}$ then $\sum_{i=1}^n A_i=E(\sum_{i=1}^n P_i)$, which is a positive contraction since $\sum_{i=1}^n P_i$ is an orthogonal projection. 

In the converse direction, note that $\sum_{i=1}^n \tau(A_i) \leq 1$. Choose mutually orthogonal projections $R_1, \cdots, R_n, R_{n+1}=I-(R_1+\cdots+R_n)$ such that $\tau(R_i) = \tau(A_i)$ for $1 \leq i \leq n$. 
Let $\bold{A} = \sum_{i=1}^m \pmb{\beta}_i \,Q_i$, where $\{Q_1,\dots,Q_m\}$ are mutually orthogonal projections that form a partition of the 
identity. Notice that 
$$
 A_j=\sum_{i=1}^m \pmb\beta_i(j)\, Q_i\ , \ \ 1\leq j\leq n \ \ \Rightarrow \ \  \sum_{j=1}^n A_j= \sum_{i=1}^m\left(\sum_{j=1}^n \pmb\beta_i(j)\right)\ Q_i\leq I\, .
$$
Let $D$ be the  $m \times (n+1)$ matrix defined by 
\[D_{ij} = \pmb{\beta}_i(j), \, 1\leq i \leq m,\, 1 \leq j \leq n,\quad D_{i,n+1} = 1 - \sum_{j=1}^{n} \pmb{\beta}_i(j),\, 1\leq i \leq m\ .\] Then $D$ has non negative coefficients, $$D\,1_{n+1}=1_m\ ,\quad q^t D=p^t \quad\text{ and } \quad D\pmb \alpha=\pmb\beta\,, $$ where 
$p$ is the column vector $(\tau(R_i))_{i=1}^{n+1}\in\R^{n+1}$, $q\in\R^m$ is the column vector $\tau(\bold Q)^t$,
 $\pmb{\alpha}$ is the $n+1 \times n$ matrix  $[e_1 \mid \cdots \mid e_n\mid 0_n]^t$ (here $\{e_i\}_{i=1}^n$ denotes the canonical basis of $\mathbb C^n$ and $0_n\in\mathbb C^n$) and $\pmb{\beta}$ is the $m \times n$ matrix  $[\pmb{\beta}_1 \mid \cdots \mid \pmb{\beta}_m]^t$. 
Hence, if we let $\bold R=(R_1,\cdots,R_n)$ then, by Proposition \ref{matrix}, we have that $\bold A\prec \bold R$.
Now, by Theorem \ref{SHThm}, we have that there is a unitary $U$ such that $E(U\,\bold{R}\,U^*) = A$. The components of $\bold P=U\,\bold{R}\,U^*$ are clearly mutually orthogonal projections.
\end{proof}

\medskip

Here is another fact close in spirit to Theorem \ref{MVCT} above.

\begin{theorem} Let $\mathcal A$ be a masa inside the type II$_1$ factor $\mathcal M$ and let $E$ denote the trace preserving conditional expectation onto $\mathcal A$. If $A\in \mathcal{A}$ is a normal contraction with finite spectrum  then, there is a unitary $U$ in $\mathcal{M}$ such that \[E(U) = A\ . \]
\end{theorem}
The unitary constructed will in fact have finite spectrum. The proof uses an argument similar to that needed for Remark \ref{Unit}, following the lines of the proof of Theorem \ref{MVCT} and we omit it. 

\subsection{Approximate multivariable Schur-Horn theorem}\label{consequences}

Analogous to the proof of \cite[Theorem 4.1.]{BhaRav}, we have:

\medskip

\begin{proof}[Proof of Theorem \ref{SHApp}]
First notice that the inclusion of the set to the left in Eq. \eqref{eqSHApp} in the set to the right is a consequence of the following two facts: on the one hand $E(U\,\bold S\,U^*)\prec \bold S$ for every unitary $U\in\mathcal M$; on the other hand the set $\{\bold A\in\mathcal A^n:\ \bold A\prec \bold S\}$ is $\|\cdot\|$-closed. 

In order to show the other inclusion, we consider the following reduction: since $\mathcal A$ is masa in $\mathcal M$ then there exists $\bold S'\in\mathcal A^n$ such that $\bold S\approx \bold S'$ so that,  $\mathcal O(\bold S)=\mathcal O(\bold S')$. Therefore $$ \overline{E(\mathcal{U}(\bold{S}))}^{||} =\overline{E(\mathcal{U}(\bold{S'}))}^{||}\,. $$
Hence, we assume (without loss of generality) that $\bold S\in\mathcal A^n$.

Let $\Delta$ be a doubly stochastic map on $\mathcal{M}$ with range in $\mathcal{A}$ so that $\Delta(\bold{S}) = \bold{A}$. Fix $\epsilon > 0$ and choose a tuple of commuting hermitians $\bold{T} = \sum_{i=1}^k \pmb{\alpha}_i P_i \in\mathcal A^n$ with finite spectrum so that 
 \[||\, \bold{S} - \bold{T}\, || < \epsilon\, .\]
 Let $\bold{B} = \Delta(\bold{T})$ and pick another tuple of hermitians in $\mathcal{A}$ with finite spectrum, $\bold{C} = \sum_{i=1}^m \pmb{\beta}_i Q_i$, such that $||\bold{B} - \bold{C}|| < \epsilon$. It is then easy to see that 
 \[||\, \bold{B} - \sum_{i=1}^m \dfrac{\tau(\bold{B}\, Q_i)}{\tau(Q_i)}\, Q_i \, || < \epsilon\,.\]
 Let $\bold{D} = \sum_{i=1}^m \dfrac{\tau(\bold{B}\, Q_i)}{\tau(Q_i)} \,Q_i$. Thus, we have that $||\,\bold{D} - \Delta(\bold{T})\,|| < \epsilon$. The operator $\tilde{\Delta}$ defined on $\mathcal{M}$ by 
 \[\tilde{\Delta}(A) :=  \sum_{i=1}^{m} \dfrac{\tau(\Delta(A) \,Q_i)}{\tau(Q_i)} \, Q_i\]
 is doubly stochastic and 
 \[\tilde{\Delta}(\bold{T})  = \sum_{i=1}^{m} \dfrac{\tau(\bold{B}\, Q_i)}{\tau(Q_i)} \, Q_i = \bold{D}\,.\]
   Thus, by Theorem \ref{SHThm} there is a unitary $U\in\mathcal M$ so that $E(U\bold{T}U^*) = \bold{D}$. We conclude that 
  \begin{eqnarray*}
 ||\,E(U\,\bold{S}\,U^*) - \bold{A}\,|| &=&  ||\,E(U\,\bold{S}\,U^*) - \Delta(\bold{S})\,||\\
 &\leq& ||\,E[U(\bold{S}-\bold{T})U^*]\,|| + ||\,E(U\,\bold{T}\,U^*) - \bold{D}\,|| + ||\,\bold{D} - \Delta(\bold{T})\,|| + ||\,\Delta(\bold{T}-\bold{S})\,|| \\
 &<& \epsilon + 0 + \epsilon + \epsilon\, .
 \end{eqnarray*}
 Since $\epsilon$ was arbitrary, the theorem follows. 
 \end{proof}

\medskip

We end this section with the proof of the approximate carpenter theorem.

\medskip

\begin{proof}[Proof of Theorem \ref{MVCApp}]
If $\bold A$ is an $n$-tuple of positive contractions in $\mathcal A$ then, by Proposition \ref{Pro}, there exists an $n$-tuple of commuting projections in $\mathcal M$, say $\bold Q$, such that $\bold A\prec\bold Q$. 
Hence, by Theorem \ref{SHApp}, given $\varepsilon >0$ there exists a unitary operator $U\in\mathcal M$ such that 
$\| E(U\,\bold Q\, U^*) - \bold A\|\leq \varepsilon$. Hence $\bold P= U\,\bold Q\, U^*$ is a commuting $n$-tuple of projections that has the desired properties.
\end{proof}

\subsection{Diagonals with finite spectra}\label{sec4}

In a recent paper \cite{DFHS}, Dykema, Hadwin, Fang and Smith presented a different approach to the Schur-Horn and carpenter theorems, relating the problems to one involving the kernels of conditional expectations onto masas, when restricted to corners of type II$_1$ factors. 
This technique allows to improve a result of Popa \cite{Popaortalg}
 on orthogonal subalgebras of type II$_1$ factors.
Recall that 
two subalgebras $\mathcal{A}$ and $\mathcal{B}$ of a type II$_1$ factor $\mathcal M$ are said to be {orthogonal} if for every $A \in \mathcal{A}$ and $B \in \mathcal{B}$, we have $\tau(AB) = \tau(A)\,\tau(B)$. This is equivalent to several other statements, most notably that the conditional expectations $E_{\mathcal{A}}$ and $E_{\mathcal{B}}$ commute as operators in $\mathcal{B}(L^2(\mathcal{M},\tau))$. Popa showed that any abelian algebra admits a $n$ dimensional orthogonal abelian algebra with all minimal projections of same trace, for any $n$. 

\begin{theorem}
Let $\mathcal{A}$ be a masa inside a type II$_1$ factor $\mathcal{M}$ and let $E$ denote the trace preserving conditional expectation onto $\mathcal A$. Suppose that for every projection $P$ in $\mathcal{M}$, $E$ restricted to $P\mathcal{M}P$ has non-trivial kernel.  Then, there is a diffuse abelian von Neumann algebra $\mathcal{B}$ in $\mathcal{M}$ that is orthogonal to $\mathcal{A}$.
\end{theorem}
\begin{proof}
We first claim that for every pair of projections $P < Q$ in $\mathcal{M}$ with $E(P) = \tau(P)I$ and $E(Q) = \tau(Q) I$ and for every $\lambda$ with $\tau(P) < \lambda < \tau(Q)$, there is a projection $R$ with $P \leq  R \leq  Q$ such that $E(R) = \lambda I$. To see this, define the set
\[\mathcal{X}:=\{S \in \mathcal{M}^{+}_1 \mid  \,P \leq S \leq Q, \,E(S) = \lambda\, I\}\, .\]
This set is weak* closed and convex; we will show that it is non-empty. Indeed, consider the set 
\[\mathcal{Y}:=\{E(S): S \in (Q-P)\mathcal{M}^{sa}_1(Q-P)\}\] which is convex and weak* precompact. Let $\mu := 2\lambda - \tau(Q+P)$; it is easy to see that 
\[\mu \in (-\tau(Q-P),\tau(Q-P)) \subset (-1,1).\]
Assume that $\mu \geq 0$ (the proof for when $\mu \leq 0$ is analogous). If we assume further that $\mu$ is not in the weak* closure of $\mathcal{Y}$, we could find an element $A \in L^{1}(\mathcal{A},\tau)$ with $\tau(A) = 1$ and a $c$ with $0 \leq c < \mu$ such that $\tau(AE(S)) = \tau(E(AS)) = \tau(AS) < c$ for each $S$ in $(Q-P)\mathcal{M}^{sa}_1(Q-P)$. But, 
 \[\tau(A(Q-P)) = \tau(E(A(Q-P))) = \tau(AE(Q-P)) = \tau(A \tau(Q-P)) = \tau(Q-P)  > \mu > c.\] This gives us a contradiction and we conclude that $\mu \in \overline{\mathcal{Y}}^{w*}$. 

There is thus a net $S_{\alpha}$ in $(Q-P)\mathcal{M}^{sa}_1(Q-P)$ such that $E(S_{\alpha}) \rightarrow \mu I$. Pick a subnet that converges to an element $S$ in $(Q-P)\mathcal{M}_1^{sa}(Q-P)$. Since the conditional expectation is weak* continuous, we have that $E(S) = \mu I$. We then see that $$P + \dfrac{S+Q-P}{2}\in \mathcal{X}\,.$$ 
 Analogous to the argument in  in \cite[Lemma 2.1]{DFHS}, we conclude that each extreme point in the weak* closed convex set $\mathcal{X}$ must be a projection. Pick one such extreme point projection and call it $R$. The claim follows. 
 

With this in hand, it is easy to see that we can iterate this argument to produce for every $n$, a tower of projections $\{P_{m,n}: 1 \leq m \leq 2^n\}$ with $\tau(P_{m,n}) = m\cdot 2^{-n}$ and with the additional property that $P_{2m-1,n}+P_{2m,n} = P_{m,n-1}$ and so that $E(P_{m,n}) = \tau(P_{m,n})I$. Let $\mathcal{B}$ be the von Neumann algebra generated by $\{P_{m,n}: 1 \leq m \leq 2^n, \, n = 1, 2, \cdots\}$. It is routine to see that $\mathcal{B}$ is diffuse and orthogonal to $\mathcal{A}$.
\end{proof}

\medskip

We can use the above result to prove a Schur-Horn theorem when the majorized tuple has finite spectrum. We first record a reduction applicable to this case.

\begin{proposition}\label{red2}
Let $\mathcal{M}$ be a type II$_1$ factor, $\mathcal{A}$ a masa in $\mathcal{M}$ and let $E$ denote the trace preserving conditional expectation onto $\mathcal A$. The following statements are equivalent:
\begin{enumerate}
\item For every tuple of hermitians $\bold{A}$  in $\mathcal{A}$ with finite spectrum and every tuple of commuting hermitians $\bold{S}$ in $\mathcal{M}$ such that $\bold{A} \prec \bold{S}$, there is a tuple $\bold{T}$ in $\mathcal{O}(\bold{S})$  such that $E(\bold{T}) = \bold{A}$. 
\item For every projection $P\in\mathcal M$ and every
tuple of commuting hermitians $\bold{S} = (S_1,\cdots,S_n)$  such that $P\,\bold S=\bold S$
in $\mathcal{M}$, there is a tuple $\bold{T}$ in $\mathcal{O}(S)$  such that $E(\bold{T}) = \frac{\tau(\bold{S})}{\tau(P)} \, P$. 
\end{enumerate}
\end{proposition}
\begin{proof}
Notice that $(1) \implies (2)$ is trivial; $(2)$ implies $(1)$ follows from a reduction argument similar to that in the proof of Theorem \ref{SHThm}, 
using the equivalence between Joint majorization and Choquet's notion of comparison of measures. \end{proof}

\medskip

This can be used to prove:
\begin{proposition}\label{Fang2}
Let $\mathcal{M}$ be a type II$_1$ factor, $\mathcal{A}$ a masa in $\mathcal{M}$ and let $E$ be the canonical conditional expectation onto $\mathcal{A}$. Suppose that for every projection $P$ in $\mathcal{M}$, the conditional expectation $E$ restricted to $P\mathcal{M}P$ has non-trivial kernel. Then, for every $n$-tuple $\bold{A} \in \mathcal{A}^n$ with finite joint spectrum and $n$-tuple $\bold{S} \in \mathcal{M}^n$, there is an element $\bold{T} \in \mathcal{O}(\bold{S})$ such that
\[E(\bold{T}) = \bold{A}\,.\]
\end{proposition}
\begin{proof}
By Proposition \ref{red2}, it is enough to prove the theorem in the case when the tuple $\bold{A}$ consists of scalars. Pick a diffuse abelian  algebra $\mathcal{B}$ orthogonal to $\mathcal{A}$ and a tuple $\bold{T}$ in $\mathcal{B}$ that is approximately unitarily equivalent to $\bold{S}$ and hence in $\mathcal{O}(\bold{S})$. Since $\mathcal{B}$ is orthogonal to $\mathcal{A}$, we see that 
\[E(\bold{T}) = \tau(\bold{S})I = \bold{A}\,.\]
\end{proof}

\medskip

The following result describes conditions on a fixed type II$_1$
 factor $\mathcal M$ and a fixed masa $\mathcal A$ in $\mathcal M$ 
that characterize when the joint majorization relations $\bold A\prec \bold S$ between tuples of commuting hermitians such that  $\bold S$ has finite joint spectrum, have a corresponding exact multivariable Schur-Horn theorem (and complements Proposition \ref{red2}). 

\begin{proposition}\label{red}
Let $\mathcal{M}$ be a type II$_1$ factor, $\mathcal{A}$ a masa in $\mathcal{M}$ and let $E$ denote the trace preserving conditional expectation onto $\mathcal A$. The following statements are equivalent:
\begin{enumerate}
\item For every tuple of hermitians $\bold{A}$ in $\mathcal{A}$ and commuting tuple of hermitians $\bold{S}$ with finite spectrum in $\mathcal{M}$ such that $\bold{A} \prec \bold{S}$, there is a unitary $U$ in $\mathcal{M}$ such that $E(U\bold{S}U^{*}) = \bold{A}$. 
\item For every tuple of hermitians $\bold{A}$  in $\mathcal{A}$ and commuting tuple of projections $\bold{P}$ in $\mathcal{M}$ such that $\bold{A} \prec \bold{P}$, there is a unitary $U$ in $\mathcal{M}$ such that $E(U\bold{P}U^{*}) = \bold{A}$. 
\item For every tuple of positive contractions $\bold{A} = (A_1,\cdots,A_n)$ in $\mathcal{A}$ such that $A_1 + \cdots + A_n \leq I$, there is a tuple $\bold{P}$ of mutually orthogonal projections in $\mathcal{M}$ such that $E(\bold{P}) = \bold{A}$.

\end{enumerate}
\end{proposition}

\begin{proof} Notice that $(1) \implies (2)$ is trivial while $(2) \implies (3)$ follows using arguments in the proof of 
Proposition \ref{Pro}: explicitely, if $\bold P=(P_1,\cdots,P_n)$ is a tuple of mutually orthogonal projections such that
$\tau(P_i)=\tau(A_i)$ for $1\leq i\leq n$ then $\bold A\prec \bold P$.

To show $(3) \implies (2)$ we consider the case $n=2$; the general case is a routine extension. Let $\Delta$ be a doubly stochastic map such that $\Delta(P_1) = A_1$ and $\Delta(P_2) = A_2$. Let $P_{12} = P_1 P_2$ and let $A_{12} = \Delta(P_{12})$. Since $A_{12} + (A_1 - A_{12}) + (A_2 - A_{12}) = A_1 + A_2 - A_{12} = \Delta(P_1 + P_2 - P_{12})$, the operator tuple $\bold{B} = (A_{12},A_1 - A_{12},A_2 - A_{12})$ satisfies the hypothesis of $(3)$ and there is a tuple of orthogonal projections $\overline{Q} = (Q_1,Q_2,Q_3)$ such that $E(\overline{Q}) = \bold{B}$. Therefore, $E[(Q_1 + Q_2,Q_1 + Q_3)] = (A_1,A_2)$ and we also have that $\tau(P_1) = \tau(Q_1 + Q_2)$, $\tau(P_2) = \tau(Q_1 + Q_3)$ and $\tau(P_1 P_2) = \tau(P_{12}) = \tau(A_{12}) = \tau(Q_1) = \tau[(Q_1 + Q_2)(Q_1 + Q_3)]$. These last facts show
that the tuples $(P_1,P_2)$ and $(Q_1 + Q_2,Q_1+Q_3)$ are unitarily equivalent. Letting $U$ be a unitary that implements this equivalence, we have that $E(U\bold{P}
U^{*}) = \bold{A}$. 

 Finally, we show $(2) \implies (1)$: let's suppose we are dealing with $n$-tuples. Write $\bold{S} = \sum_{k=1}^m \pmb{\alpha_k}P_k$ and let $\Delta$ be the doubly stochastic map such that $\Delta(\bold{S}) = \bold{A}$. After composing with a conditional expectation, we may assume that the range of $\Delta$ lies in $\mathcal{A}$. We have that $\bold{B} = (\Delta(P_1),\cdots,\Delta(P_m))\prec \bold{P} = (P_1,\cdots,P_m)$ and hence, by $(2)$, there is a unitary $U$ such that $E(U\bold{P}U^{*}) = \bold{B}$. If we let $X$ be the $n \times m$ matrix given by $X_{ij} = \pmb{\alpha_j}(i)$, we see that $\bold{S} = X \bold{P}^t$ and that $\bold{A} = X \bold{B}^t$. Thus,
 \[E(U\bold{S} U^{*}) = E(U X\bold{P}^tU^{*}) = X E(U\bold{P}U^{*})^t = X\bold{B}^t = \bold{A}     \,.\]
\end{proof}

\section[Discrete factors]{Multivariable Schur-Horn theorems in discrete factors}

At the beginning of section \ref{obs} we considered an example (for $3\times 3$ normal matrices) which showed that the natural extension of the classical Schur-Horn theorem (for hermitian matrices) 
to normal matrices fails. In this section we show that it is possible to dilate such a problem and obtain a positive result.
On the other hand, we investigate the approximate diagonals of commuting $n$-tuples of selfadjoint operators in $B(\mathcal H)$ with finite spectrum (all with infinite multiplicity). In particular, we show that in this context there are no non-trivial obstructions for a sequence in $\mathbb R^n$ to be an approximate diagonal. This contrasts with the results in 
\cite{ArDiPN} on exact diagonals of normal operators with finite spectrum. We end this section with some final remarks.

\subsection{Arveson's example revisited}\label{secDS}

Throughout this section $\mathcal M_d$ denotes the algebra of $d\times d$ complex matrices, $\mathcal A_d\subset \mathcal M_d$ the masa of diagonal matrices and $E_d:\mathcal M_d\rightarrow \mathcal M_d$ denotes the compression to the diagonal i.e. the trace preserving conditional expectation onto $\mathcal A_d$.

Recall that in the matrix context, given normal matrices $A,\,S\in \mathcal M_d$ with eigenvalues $\pmb\mu,\,\pmb\lambda\in \mathbb C^d$ respectively then $A\prec S$ (or equivalently the commuting two tuple $(\text{Re}(A),\text{Im}(A))$ is majorized by the commuting two tuple  $(\text{Re}(S),\text{Im}(S))$) if and only if there exists an $d\times d$ doubly stochastic matrix $D$ such that $D\pmb\lambda=\pmb\mu$.
 
Let us recall Arveson's example from section \ref{obs}. There we considered the normal matrices $A,\,S\in\mathcal M_3$ with eigenvalues $\pmb\mu = (\dfrac{1}{2},\dfrac{i}{2},\dfrac{1+i}{2})$ and $\pmb\lambda =(1,0,i)$; if we let $D$ be the doubly stochastic matrix in Eq. \eqref{exa ds not uni}
then, we have that $D\pmb\lambda = \pmb\mu$. However, there is no unitary matrix $U\in\mathcal M_3$ so that $E_3(USU^*) = A$ or equivalently the multivariable Schur-Horn theorem in $\mathcal M_3$ fails.
 Nevertheless, it is fairly simple to compute examples of unitary matrices $U\in \mathcal M_6$ such that the normal matrix $U^*(N\oplus N)\,U$ has main diagonal $\pmb\mu\oplus \pmb \mu\in\C^6$. This suggests that we could get a Schur-Horn type result for a suitable dilation of the original problem.

In order to put the previous claim in context, in what follows we consider the unital $*$-subalgebra $\mathcal{A}_m\otimes \mathcal M_d\subset \mathcal M_m \otimes \mathcal M_d$ together with its associated trace preserving conditional expectation $E_{\mathcal{A}_m\otimes \mathcal M_d}$ determined by $$E_{\mathcal{A}_m\otimes \mathcal M_d}(A\otimes B)=E_m(A)\otimes B\, .$$
Also, given $\alpha\in\C^m$ we denote by $\operatorname{diag}(\alpha) \in \mathcal{A}_m$ the diagonal matrix with main diagonal $\alpha$.

\begin{proposition}\label{pro0}
Let $\Delta:\mathcal M_d\rightarrow \mathcal M_d$ be given by $\Delta(B)=\sum_{j=1}^m \alpha_j\, U_j \,B \, U_j^*$, where $U_j\in \mathcal M_d$, $\|U_j\|\leq 1$ for $1\leq j\leq m$ and $\pmb\alpha=(\alpha_j)_{j=1}^m$ are coefficients for a convex combination. Then, there exists $U\in \mathcal M_{m}\otimes \mathcal M_d$ with $\|U\|\leq 1$, such that for every $B\in\mathcal  M_d$ we have that
\begin{eqnarray}\label{eq pro c}
 E_{\mathcal{A}_m\otimes \mathcal M_d} (U\, ( \operatorname{diag}(\pmb\alpha)\otimes B )\,U^*)=\frac{1}{m} \, I_m\otimes \Delta(B)\,.
  \end{eqnarray} 
 Moreover, if $U_j\in\mathcal M_d$ is unitary for $1\leq j\leq m$ then we can chose a unitary $U\in\mathcal M_{m}\otimes \mathcal M_{d}$ satisfying Eq. \eqref{eq pro c}. 
\end{proposition}
\begin{proof} Consider $\omega=\exp(\frac{2\,\pi\,i}{m})$ and let  $U\in \mathcal M_{m}\otimes \mathcal M_{d}$ be given by 
\begin{equation}\label{eq uustar}   U= \frac{1}{\sqrt{m}} \ \sum_{j,\,k=1}^m    \omega^{(j\cdot k)}\ E_{jk}\otimes  U_k \ ,\end{equation}
where $\{E_{jk}\}_{j,k=1}^m$ denotes the canonical basis of $\mathcal M_{m}$. Using that $\sum_{j=1}^m \omega^{\,j\cdot k}=0$ for every $1\leq |k|\leq m-1$ we get that 
\begin{eqnarray*}
U^*U&=&
\frac{1}{m} \ \sum_{k=1}^m\ E_{kk}\otimes  U_k^* \, U_{k}\,. 
\end{eqnarray*}  
In particular $\|U\|\leq 1$. On the other hand, if $B\in \mathcal M_d$ then
\begin{eqnarray*}
U\,(\text{diag}(\pmb\alpha)\otimes B)\,U^* &=& \frac{1}{m} \sum_{j,\,k\,,j',k'=1}^m  \omega^{k\cdot j- k'\cdot j'} E_{jk}\,\text{diag}(\pmb\alpha)\, E_{k'j'}\otimes U_k B \,U_{k'}^*\\
&=& \frac{1}{m} \sum_{j,\,k,\,j'=1}^m  \omega^{k\cdot(j- j')} \alpha_k \, E_{jj'} \otimes U_k B \,U_{k}^*\ .
\end{eqnarray*}
Using the computations above we see that  
\[E_{\mathcal{A}_m\otimes \mathcal M_d}( U\,(\text{diag}(\pmb\alpha)\otimes B)\,U^*)=\frac{1}{m} \sum_{j,k=1}^m  E_{jj} \otimes \alpha_k \, U_k B \,U_{k}^*=\frac{1}{m} I_m\otimes \Delta(B).\]
 Finally, assume that $U_j\in\mathcal M_d$ is unitary for $1\leq j\leq m$. Then, if $U$ is defined as in Eq. \eqref{eq uustar}, the computations above show that $U^*U=I_{m\cdot n}$ and hence $U\in\mathcal 
 M_{m\cdot d}$ is a unitary matrix.
 \end{proof}

\begin{corollary}\label{cor:inflation of dsmaps}
Let $D\in\mathcal M_d$ be a doubly stochastic matrix. Then, there exists $m\in\mathbb N$, coefficients for a convex combination $\pmb\alpha=(\alpha_j)_{j=1}^m$	and a unitary $U\in\mathcal M_{m}\otimes \mathcal M_{n}=\mathcal M_{m\cdot n}$ such that for every 
$\pmb\beta\in\mathbb C^d$
\begin{equation}\label{eq: inflation of dsmaps}
E_{m\cdot d}(U\,(\text{diag}(\pmb \alpha)\otimes \text{diag}(\pmb \beta))\,U^*)=\frac{1}{m}\ I_m\otimes \text{diag}(D\,\pmb \beta)\,.
\end{equation}In particular, if $\alpha_j=\frac{1}{m}$ for $1\leq j\leq m$ then we get that 
for every 
$\pmb\beta\in\mathbb C^d$
\begin{equation}\label{eq: inflation of dsmaps2}
E_{m\cdot d}(U\,(I_m\otimes \text{diag}(\pmb \beta))\,U^*)= I_m\otimes \text{diag}(D\,\pmb \beta)\,.
\end{equation}
\end{corollary}
\begin{proof} By Birkhoff's theorem (see \cite{BhaBoo}) there exists coefficients for a convex combination $\pmb\alpha=(\alpha_j)_{j=1}^m$ and permutation matrices $P_1,\cdots,P_m\in\mathcal M_d$ such that $D=\sum_{j=1}^m \alpha_j\,P_j$. 
Let $\Delta:M_d\rightarrow \mathcal M_d$ be given by $\Delta(B)=\sum_{j=1}^m \alpha_j\, P_j\, B\, P_j^*$.
Using the fact that $P_j\,\text{diag}(\pmb \beta)\,P_j^*=\text{diag}(P_j\,\pmb \beta)$ we see that 
\begin{equation}\label{eq. delta dsmap}
\Delta(\text{diag}(\pmb \beta))=\text{diag}(D\,\pmb \beta) \quad \text{for}\quad \pmb \beta\in\mathbb C^d\,.
\end{equation}
By construction of $\Delta$ we can apply Proposition \ref{pro0} and get a unitary $U\in\mathcal M_m\otimes \mathcal M_d=\mathcal M_{m\cdot d}$ for which Eq. \eqref{eq pro c} holds. Hence, using that $E_{m\cdot d}(E_{\mathcal{A}_m\otimes \mathcal M_d}(C))=E_{m\cdot d}(C)$ for $C\in\mathcal M_{m\cdot d}$, together with Eq. \eqref{eq. delta dsmap} we see that Eq. \eqref{eq: inflation of dsmaps} holds. The second part of the statement is a immediate consequence of the previous facts.
\end{proof}

\medskip

We now consider again Arveson's example described at the beginning of this section.
Let $D\in DS(3)$ be defined as in Eq. \eqref{exa ds not uni} and consider the normal matrices $A,\,S\in\mathcal M_3$ with eigenvalues $\pmb\mu = (\dfrac{1}{2},\dfrac{i}{2},\dfrac{1+i}{2})$ and $\pmb\lambda =(1,0,i)$; hence $D\pmb\lambda=\pmb\mu$.
It is easy to see that in this case there exist permutation matrices $P_1,\,P_2\in \mathcal M_3$ such that $D=1/2\cdot (P_1+P_2)$. Hence, by Corollary \ref{cor:inflation of dsmaps} 
there exists a unitary $U\in \mathcal M_6$ such that 
$$ E_{6}(U \, (I_2\otimes \text{diag}(\pmb\lambda)) \, U^* )=I_2\otimes \text{diag}(\pmb\mu)$$
This last fact explains why we can obtain a Schur-Horn theorem for a dilation of the original problem as claimed.

\begin{remark}\label{rem: not the gen case} Corollary \ref{cor:inflation of dsmaps} shows that there are doubly stochastic matrices $D\in \mathcal M_d$ for which there exist $m\geq 1$ and a unitary $U\in\mathcal M_{m\cdot d}$ such that Eq. 
\eqref{eq: inflation of dsmaps2} holds, for every $\pmb\beta\in\C^d$. 
It is natural to wonder whether this is the general case i.e., given an arbitrary $d\times d$ doubly stochastic matrix $D$ we can ask
whether there always exist $m\geq 1$ and a unitary $U\in \mathcal M_{m\cdot d}$ such that Eq. \ref{eq: inflation of dsmaps2} holds for every $\pmb\beta\in\C^d$. 
It turns out that this is not the case.
 Indeed, consider the following variation of Arveson's example:  fix positive irrationals $a$ and $b$ adding up to one and let $D$ be the doubly stochastic matrix given by
$$D = \left( \begin{array}{ccc}
a & b & 0\\
0 & a & b \\ 
b & 0 & a \end{array} \right)\ .$$ 
\end{remark}

\begin{proposition} With the notations of Remark \ref{rem: not the gen case}, there is no $m\geq 1$ so that there exists a unitary $U\in \mathcal M_{3m}$ such that $$ E_{3m}(U^* (I_m\otimes \text{diag}(\pmb\lambda))\,U )=I_m\otimes \text{diag}(D \pmb\lambda) \ \text{ for every } \pmb\lambda\in \C^3\, .$$
\end{proposition}

\begin{proof}
Suppose there is in fact a matrix $U\in \u(3m)$ such that \[ E_{\mathcal{A}_m\otimes \mathcal A_3}(U^* (I_m\otimes \text{diag}(\lambda))\,U )=I_m\otimes \text{diag}(D \lambda)\] 
Let us write 
\[U = \left( \begin{array}{ccc}
U_{11} & U_{12} & U_{13}\\
U_{21} & U_{22} & U_{23} \\ 
U_{31} & U_{32} & U_{33} \end{array}\right) \]
A simple matrix calculation shows then that 
\[ D = \left( \begin{array}{ccc}
D_{11} & D_{12} & D_{13}\\
D_{21} & D_{22} & D_{23} \\ 
D_{31} & D_{32} & D_{33} \end{array}\right) \left( \begin{array}{c} 
\tilde{\alpha}\\
\tilde{\beta}\\
\tilde{\gamma}\end{array}\right) =  \left( \begin{array}{c} 
\tilde{x}\\
\tilde{y}\\
\tilde{z}\end{array}\right)\]
where $D(\alpha,\beta,\gamma) = (x,y,z)$, where $\tilde{\alpha}$ represents the vector $\left( \begin{array}{c} 
\alpha\\
\vdots\\
\alpha\end{array}\right)$ and $D_{ij}$ is the matrix given by $D_{ij}(k,l) = |U_{ij}(k,l)|^2$. We have,
\[ D \left( \begin{array}{c} 
\tilde{1}\\
\tilde{0}\\
\tilde{0}\end{array}\right) =  \left( \begin{array}{c} 
\tilde{a}\\
\tilde{0}\\
\tilde{b}\end{array}\right)\quad \text{and} \quad D \left( \begin{array}{c} 
\tilde{0}\\
\tilde{1}\\
\tilde{0}\end{array}\right) =  \left( \begin{array}{c} 
\tilde{b}\\
\tilde{a}\\
\tilde{0}\end{array}\right) \quad \text{ and } \quad  D \left( \begin{array}{c} 
\tilde{0}\\
\tilde{0}\\
\tilde{1}\end{array}\right) =  \left( \begin{array}{c} 
\tilde{0}\\
\tilde{b}\\
\tilde{a}\end{array}\right)\]
In particular, we have
\[\operatorname{Tr}(U_{11} U_{11}^{*}) = \sum_{k,l}|U_{11}(k,l)|^2 = \sum_{k,l}D_{11}(k,l) = m\alpha\]
Similarly,
\[\operatorname{Tr}(U_{31} U_{31}^{*}) = \sum_{k,l}|U_{31}(k,l)|^2 = \sum_{k,l}D_{31}(k,l) = m\beta\]
Now note that $D_{13}$ and $D_{32}$ and hence $U_{13}$ and $U_{32}$ are zero. 
We thus have that $U_{11}U_{31}^{*} = 0$. We also have that $U_{11}^{*}U_{11} + U_{31}^{*}U_{31} = I$. We conclude from Lemma \ref{Irr} below that $U_{11}$ is a partial isometry and hence $\operatorname{Tr}(U_{11}^{*}U_{11}) = m\alpha$ is an integer. Thus, $\alpha$ cannot be irrational. 
\end{proof}

\begin{lemma}\label{Irr}
 Suppose $A$ and $B$ are in $\m_d$ so that $A^{*}A + B^{*}B = I$ and $AB^{*} = 0$. Then, $A$ and $B$ are partial isometries.
\end{lemma}
\begin{proof}
 Letting $A = UH$ and $B = VK$ be the polar decompositions, we have that $H^2 + K^2 = I$ and $UHKV^{*} = 0$. Thus,
 \[HK = U^{*}U HK V^{*}V = 0\]
 The two equations $H^2 + K^2 = I$ and $HK = 0$ imply that $H$ and $K$ are projections and hence, that $A$ and $B$ are partial isometries.
\end{proof}

\begin{remark}\label{rem: final comments on Arveson's example}
We point out that using an argument similar to that in the proof of Corollary \ref{cor:inflation of dsmaps}
we can prove the following: 
given a doubly stochastic matrix $D\in\mathcal M_d$  and $\epsilon>0$, there exist $m\in\N$ and a unitary $U\in\mathcal M_{m\cdot d}$ such that, for every $\pmb\beta\in\C^d$ we have that 
\begin{equation}\label{eq cor1} \|E_{m\cdot d}(U \,(I_m \otimes \text{diag}(\pmb \beta))\,U^*) - I_m\otimes \text{diag}(D\cdot \pmb\beta)\|\leq \epsilon \, \|\text{diag}(\pmb \beta)\|\, . 
\end{equation}The proof is straightforward and we omit it.
\end{remark}

\subsection{Approximate multivariable Schur-Horn theorems in $\mathcal B(\mathcal H)$}\label{MSAppB(H)}
There has been considerable progress on a related problem - that of characterizing the diagonals of operators with finite spectrum in $\mathcal{B}(\mathcal{H})$. Analysis of this fundamental problem was initiated by Kadison in \cite{KPNAS1,KPNAS2} who gave a complete characterization of the diagonals of projections. Kadison's results have recently been extended by Bownik and Jasper \cite{BowJas2,BowJas1}, who have completely described the diagonals of self-adjoint operators with finite spectrum. 

The multivariable case, where we ask for a description of the joint diagonals of a tuple of commuting hermitian operators is currently only partially 	
understood, even in its most simple formulation. Study of this problem was initiated by Arveson in \cite{ArDiPN}, who analyzed the possible diagonals of normal operators with finite spectrum, all of infinite multiplicity. He discovered an index obstruction analogous to Kadison's index for projections from \cite{KPNAS2}. He also pointed out that there are other obstructions: however the ones he discovered stem from the fact that it is not possible to have a multivariable Schur-Horn theorem in matrix algebras.

Let $\mathcal{A} \subset \mathcal{B}(\mathcal{H})$ be an atomic masa and let $\bold{S}$ be an $n$-tuple of commuting hermitians in $\mathcal{B}(\mathcal{H})$ such that the joint spectrum consists of finitely many points, all of infinite multiplicity. We are interested in the set of joint diagonals of the $n$-tuple $\bold S$, i.e. $E(\mathcal U_\mathcal H(\bold S))=\{E(U\bold{S}U^{*}): U \in \mathcal{U}(\mathcal{H})\}$, where $\mathcal U(\mathcal H)$ is the group of unitary operators acting on $\mathcal B(\mathcal H)$ and $E$ is the normal conditional expectation onto $\mathcal{A}$. A precise characterization of $E(\mathcal U_\mathcal H(\bold S))$ appears quite challenging. 

Let us briefly recall the setting and results from \cite{ArDiPN}. With the notations above, let $\mathcal{X} = \{\pmb{\lambda}_1,\cdots,\pmb{\lambda}_k\}\subset \mathbb R^n$  be the joint spectrum of the $n$-tuple $\bold{S}$ and let $\bold{D} = \{\bold{d}_1,\bold{d}_2,\cdots\}$ be a sequence in $\mathbb R^n$. Assume further that $\mathcal{X}$ is the set of vertices of a convex polygon $C_\mathcal X$ in $\mathbb R^n$. Arveson considered what he called the {\it critical sequences} $\bold{D}$ that accumulate rapidly in $\mathcal X$, i.e. those 
sequences in $C_\mathcal X$ for which 
 there is a map $\phi: \mathbb{N} \rightarrow \{1,\cdots,k\}$ so that 
\[\sum_{m=1}^{\infty} ||\pmb{\lambda}_{\phi(m)} - \bold{d}_{m}|| < \infty\, .\]
In this case Arveson showed that if $\bold D\in E(\mathcal U_\mathcal H(\bold S))$ then there exist 
$\nu_1,\cdots,\nu_k\in\mathbb Z$ such that $\sum_{j=1}^k \nu_j=0$ and $\sum_{m=1}^{\infty} \pmb{\lambda}_{\phi(m)} - \bold{d}_{m} =\sum_{j=1}^k \nu_j\,\pmb\lambda_j$.

In what follows, we look at the ``non-summable'' case where we assume that the sequence is in a closed subset of the interior of $C_\mathcal X$. We then use this non-summable case to characterize the set of approximate diagonals 
$$\overline{E(\mathcal U_\mathcal H(\bold S))}^{||} \subset \mathcal A\, .$$
As we shall see, the index obstructions disappear as long as we are only interested in approximate diagonals.
We begin our analysis with the following simple fact. 

\begin{lemma}\label{const}
Let $\bold{S}$ be a commuting tuple of hermitians in $M_{n}(\mathbb{C})$. Then, there is a unitary $U$ so that $E(U\bold{S}U^{*}) = n^{-1}\,\operatorname{tr}(\bold{S})\, I$.
\end{lemma}
\begin{proof}
We may pick a unitary $U$ so that $U\bold{S}U^{*}$ is diagonal. Let $V$ be the Fourier unitary, $V_{ij} = n^{-1/2}\,{\omega^{(i-1)(j-1)}}$ where $\omega = \operatorname{exp}(\frac{2\pi i}{m})$. It is elementary to see that  the diagonal of $VDV^{*}$ is constant and equal to the normalised trace of $D$ for any diagonal matrix $D$. Hence, the diagonal of $VU\bold{S}U^{*}V^{*}$ is $n^{-1}\operatorname{tr}(\bold{S})I$. 
\end{proof}

\medskip 


\begin{remark}\label{rem: notation for lemma int}In the statement of the next lemma we use the following terminology: let $\mathcal{X} = \{\pmb{\lambda}_1,\cdots,\pmb{\lambda}_k\}$ be the set of vertices of the convex polygon $C_\mathcal X\subset \R^n$. It is clear that $C_\mathcal X$ coincides with the convex hull of $\mathcal X$ in $\mathbb R^n$. In this case there exists a unique subspace $V_\mathcal X\subset \mathbb R^n$ such that $C_\mathcal X\subset V_\mathcal X$ and $C_\mathcal X$ has nonempty interior, $C_\mathcal X^0=C_\mathcal X^{0(V_\mathcal X)}\neq \emptyset$, relative to $V_\mathcal X$.
\end{remark}

\begin{lemma}\label{dec} Let $\mathcal{X} = \{\pmb{\lambda}_1,\cdots,\pmb{\lambda}_k\}$ be the set of vertices of the convex polygon $C_\mathcal X\subset \R^n$ and let $\bold{d}$ be a point in the interior of $C_\mathcal X$. 
Then, we can find a decomposition,  
\[\bold{d} = \sum_{1 \leq i\neq j \leq k} q_{ij}\ (\alpha_{ij} \ \pmb{\lambda}_i + \beta_{ij}\  \pmb{\lambda}_{j})\]
where the $\{q_{ij}: 1 \leq i \neq j \leq n\}$ are rational numbers summing upto $1$, $\alpha_{ij}\neq 0\neq \beta_{ij}$ and 
$\alpha_{ij}+ \beta_{ij} =1$ whenever $q_{ij}\neq 0$ and such that for every $1\leq i\leq k$ there exists $1\leq j\neq i\leq k$ such that $q_{ij}\neq 0$.
  \end{lemma}
\begin{proof} By Remark \ref{rem: notation for lemma int} we see that we can assume that $n=k-1$. Moreover, by considering an appropriate affine linear transformation we can assume, without loss of generality, that $\pmb \lambda_1=\pmb 0$ and $\pmb \lambda_j=\pmb e_{j-1}$ for $2\leq j\leq k$, where $\{\pmb e_j\}_{j=1}^{k-1}$ is the canonical basis of 
$\mathbb R^{k-1}$.

We now argue by induction on $k\geq 2$. When $k = 2$, there is $\alpha\in (0,1)$ such that $\bold{d} = \alpha \,\pmb{0}_1 + (1-\alpha)\, \pmb{e}_1$ and we see that the desired decomposition 
 is achieved. Now, assume that the theorem has been proved for $k = 2, \cdots, K-1$. Let $\bold{Y}$ be the convex hull of the first $K-1$ points $\{\pmb \lambda_1, \cdots, \pmb{\lambda}_{K-1}\}=\{\pmb 0,\, \pmb{e}_1, \cdots, \pmb{e}_{K-2}\}$. 

Let the line segment starting at $\pmb{\lambda}_{K}=\pmb{e}_{K-1}$ and passing through $\bold{d}$ meet $\bold{Y}$ in $\bold{f}$. Since $\bold{d}$ was in the interior of the convex hull of $\mathcal{X}$, we have that $\bold{f}$ is in the interior of $\bold Y$ (as described in Remark \ref{rem: notation for lemma int}).
It is clear that there is a $\epsilon$ so that for every $a \in (1-\epsilon,1)$, the line segment starting from $a \ \pmb{e}_{K-1} + (1-a) \ \pmb{0}$ and passing through $\bold{d}$ meets $\bold{Y}$ in its interior as well; call the latter point $\bold{f}_a$. 

It is elementary to see that we may pick an $a \in (1-\epsilon,1)$ so that $\bold{d}$ is a non-trivial \emph{rational} convex combination of $a \ \pmb{e}_{K-1} + (1-a) \ \pmb{0}$ and $\bold{f}_a$, i.e.
\[\bold{d} = q\ (a \ \pmb{e}_{K-1} + (1-a)\ \pmb{0}) + (1-q)\ \bold{f}_a \ , \quad \text{for} \ q\in \mathbb Q\cap (0,1)\,.\]
Now, by the induction hypothesis we have 
\[\bold{f}_a = \sum_{1 \leq i\neq j \leq K-1} q_{ij} \ (\alpha_{ij}\ \pmb{\lambda}_i + \beta_{ij} \ \pmb{\lambda}_{j})\,,\]
with the coefficients $q_{ij}$, $\alpha_{ij}$ and $\beta_{ij}$ for $1\leq i\neq j\leq K-1$ satisfying the properties of the statement.
Thus, we see that 
\[\bold{d} = q\ (a\  \pmb{\lambda}_K + (1-a)\  \pmb{\lambda}_1) + \sum_{1 \leq i\neq j \leq K-1} (1-q)\ q_{ij}\ (\alpha_{ij}\ \pmb{\lambda}_i + \beta_{ij} \ \pmb{\lambda}_{j})\]
and we have our desired decomposition. 
\end{proof}

\begin{proposition}
Let $\mathcal{A}$ be an atomic masa in $\mathcal{B}(\mathcal{H})$ and let $E$ denote the trace preserving conditional expectation onto $\mathcal A$. Let $\bold{S}$ be a tuple of commuting hermitians with finite joint spectrum where each joint eigenvalue has infinite multiplicity. Let $D=\{\bold{d}_1,\bold{d}_2,\cdots\}$ be a sequence with only finitely many distinct entries, all lying in the interior of the convex hull of the joint spectrum of $\bold{S}$. Then, there is a unitary $U$ so that $E(U\bold{S}U^{*}) = \bold{D}$.
\end{proposition}
\begin{proof} Let us write the joint spectrum as $\mathcal{X} = \{\pmb{\lambda}_1,\cdots,\pmb{\lambda}_k\}$. We prove the proposition in three steps. 

\noindent \textbf{Step 1: } First of all, we assume that the diagonal sequence is constant, $\bold{D}= \{\bold{d},\bold{d},\cdots\}$. We use Lemma \ref{dec} to find a decomposition,  
\begin{equation}\label{Eq: decom d2}
\bold{d} = \sum_{(i,j)\in I} q_{ij}\ (\alpha_{ij} \ \pmb{\lambda}_i + \beta_{ij} \ \pmb{\lambda}_{j})
\end{equation}
where the $I\subset \{(i,j): \ 1 \leq i \neq j \leq k\}$, 
$\{q_{ij}: \ (i,j)\in I\}$ are non-zero rational numbers summing upto $1$, $\alpha_{ij}\neq 0\neq \beta_{ij}$ and 
$\alpha_{ij}+ \beta_{ij} =1$ for $(i,j)\in I$ and such that for every $1\leq i\leq k$ there exists $1\leq j\neq i\leq k$ such that $(i,j)\in I$.

Write the Hilbert space $\mathcal{H}$ as $\oplus_{(i,j)\in I} \mathcal{H}_{ij}$ where each of the spaces $\mathcal{H}_{ij}$ is infinite dimensional. Let $\mathcal{F}_{ij} := \{e^{ij}_n : n = 1, 2, \cdots\}$ be a basis for $\mathcal{H}_{ij}$. For $(i,j)\in I$ let $\pmb{\mu}_{ij} := \alpha_{ij}\  \pmb{\lambda}_i + \beta_{ij} \ \pmb{\lambda}_{j}$ and let 
$P_{ij}\in \mathcal B(\mathcal H_{ij})$ be a infinite projection such that $I_{ij}-P_{ij}\in \mathcal B(\mathcal H_{ij})$ is also infinite. Now, let 
$\bold{S}_{ij}=\pmb \lambda_i\ P_{ij}+ \pmb \lambda_j\ (I_{ij}-P_{ij})\in \mathcal B(\mathcal H_{ij})$ which is a commuting $n$-tuple of hermitian operators with joint spectrum consisting of $\pmb{\lambda}_{i}$ and $\pmb{\lambda}_j$, both of infinite multiplicity. Because of the properties of the index set $I$ it follows that the operator $\oplus_{(i,j)\in I} \bold{S}_{ij}$ is unitarily equivalent to $\bold{S}$. Hence, without loss of generality we can assume that  
$\bold{S}=\oplus_{(i,j)\in I} \bold{S}_{ij}$.

Let $E_{ij}$ denote the compression to the diagonal with respect to the orthonomal basis $\mathcal F_{ij}$ in $\mathcal B(\mathcal H_{ij})$. Since $\alpha_{ij}\in (0,1)$ for $(i,j)\in I$ then, by  Kadison's Pythagorean theorem \cite{KPNAS2}, there exists a unitary $U_{ij}\in \mathcal B(\mathcal H_{ij})$ such that $$E_{ij}(U_{ij}\ P_{ij}\ U_{ij}^*)=\alpha_{ij}\, I_{ij} \ \implies E_{ij}(U_{ij}\ \bold S_{ij}\ U_{ij}^*)=\pmb \mu_{ij}\, I_{ij}\ , \quad \text{for} \  (i,j)\in I\,.$$ 
Let use write $U = \oplus_{(i,j)\in I} U_{ij}$ and notice that $U\ \bold{S}\ U^{*} = \oplus_{(i,j)\in I} U_{ij} \bold S_{ij}\ U_{ij}^*$. 


Let $N$ be a natural number such that $q_{ij}=n_{ij}\, N^{-1}$ with $n_{ij}$ a positive integer, for $(i,j)\in I$. Notice that $\sum_{(i,j)\in I}n_{ij}=N$.
Using Eq. \eqref{Eq: decom d2}
 we may write 
\begin{equation}\label{Eq: norm trace1}
\bold{d} = \dfrac{1}{N}\sum_{(i,j)\in I} n_{ij} \ \pmb{\mu}_{ij}\, .
\end{equation}

Now, we take a different decomposition of $\mathcal{H}$: let $\mathcal{K}_m$ be the finite dimensional subspace (of dimension $N$) spanned by $\{ e^{ij}_{l}: (i,j)\in I \, , \ 1+(m-1)\  n_{ij}\leq l \leq m\ n_{ij}\}$. It is clear that $\mathcal{H} = \oplus_{m\in\mathbb N} \mathcal{K}_m$. Letting $P_m$ be the orthogonal projection onto $\mathcal{K}_m$, we see that the matrix $P_m (U\,S\,U^{*}) P_m$ has diagonal consisting of $n_{ij}$ copies of $\pmb{\mu}_{ij}$ for $(i,j)\in I$ and hence, has normalized trace (see Eq. \eqref{Eq: norm trace1}) equal to $\bold d$.
By Lemma \ref{const}, one may conjugate by a unitary matrix $V_m$ to make the diagonal constant, this constant being $\bold{d}$. Letting $V$ be the unitary $\oplus_{m\in\mathbb N} V_m$, we see that $E(VU\bold{S}U^{*}V^{*}) = \{\bold{d}, \bold{d}, \cdots\}$, where $E$ is the conditional expectation onto the diagonal algebra associated to the basis $\mathcal{F}=\{e^{ij}_n:\ (i,j)\in I\, , \ n\in\mathbb N\}$.

\noindent \textbf{Step 2:} We next assume that the diagonal consists of a point $\bold{d}$ repeated infinitely often and a single other point $\bold{e}$. Pick a point $\bold{f}$ in the interior of the convex hull of $\mathcal{X}$ so that $\bold{d}$ is a rational convex combination of $\bold{e}$ and $\bold{f}$. Applying the argument in Step1 twice, we see that there is a unitary $U$ so that 
\[U\bold{S}U^{*} = \left( \begin{array}{ccc}
\bold{e} & \ast & 0\\
\ast & \bold{A} &  0\\
0 &  0 & \bold{B}  \end{array} \right)\, , \]
were we consider the  3$\times$3 block representation with respect to the orthogonal decomposition 
$\mathcal{H} = \mathbb{C} \oplus \mathcal{H}_1 \oplus \mathcal{H}_2$ so that $\mathcal{H}_1$ and $\mathcal{H}_2$ are infinite dimensional and $\bold{A}$ and $\bold{B}$ have constant diagonals $\bold{e}$ and $\bold{f}$ respectively. We may write $$\bold{d} = \dfrac{a}{a+b}\ \bold{e} + \dfrac{b}{a+b}\ \bold{f}$$ for some positive integers $a, b$. Choose projections $P_n$ whose range space is spanned by basis vectors, numbering $a$ from $\mathcal{H}_1$ and $b$ from $\mathcal{H}_2$ and so that $\oplus_{n\in\mathbb N} P_n \mathcal{H}= \mathcal{H}_1 \oplus \mathcal{H}_2$.  Notice that the normalized trace of $P_nU\bold{S}U^{*}P_n$ is $\bold{d}$; hence, by Lemma \ref{const}, we may conjugate by a unitary to achieve a constant diagonal $\bold{d}$ on $P_n \mathcal{H}$. Taking direct sums, we see that there is a unitary conjugate with diagonal $\{\bold{e}, \bold{d}, \bold{d}, \cdots\}$. 

\noindent \textbf{Step 3:} This argument can be trivially extended, by taking direct sums, to the case when the diagonal consists of only finitely many distinct elements - In the case an element occurs infinitely many times, use the first part of the proof; in the case it occurs finitely many times, the second part. The proposition follows.
\end{proof}

\medskip

\begin{proof}[ Proof of Theorem \ref{thm on appr diag of normals}]
%
Any sequence in $\mathbb R^n$ can be approximated in the $\ell^{\infty}$ norm by a sequence with only finitely many distinct entries. A moment's thought shows that if the sequence belongs to the convex hull of $\mathcal{X}$, the approximating sequence with finitely many distinct entries can be chosen to lie in the interior of the convex hull.
\end{proof}

\medskip

Some final remarks are in order. We have shown in Theorem \ref{SHApp} that in the type II$_1$ setting there are nice characterizations of the set of approximate joint diagonals of a tuple of commuting hermitians operators in terms of joint majorization. This result is somewhat unexpected, since its finite dimensional version (i.e. for the finite discrete factor $\m_d$) fails.
Yet, the structure of the set of joint diagonals of a tuple of commuting hermitians operators 
in a type II$_1$ factor still remains to be understood as there are obstructions for a full extension of the 
Schur-Horn theorem for selfadjoint operators to the multivariable setting in a type II$_1$ factor (as explained in Remark \ref{rem diag vs approx diag}). The fact that there are differences between joint diagonals and approximate joint diagonals of tuples of commuting hermitians
is in accordance with the distinction between diagonals and approximate diagonals of selfadjoint and normal operators in $B(\mathcal H)$ (with respect to discrete masas), as seen by comparing the work of Neumann \cite{NeuSH} and Theorem \ref{thm on appr diag of normals} above (on approximate diagonals) with the work of Kadison \cite{KPNAS2}, Bownik and Jasper \cite{BowJas2,BowJas1}, Arveson and Kadison \cite{ArvKad} and Arveson \cite{ArDiPN} (on diagonals).
This remarks lead to what seems to be a challenging problem, namely to determine the nature of the 
obstructions in the type II$_1$ setting.
%
%
%
%
%

\Addresses

\end{document}